\documentclass[11pt]{article}
\usepackage{float}

\usepackage[T1]{fontenc}
\usepackage[utf8]{inputenc}
\usepackage[a4paper,margin=1in]{geometry}
\usepackage{lmodern}
\usepackage{amsmath,amssymb,amsthm,mathtools}
\usepackage{graphicx}
\graphicspath{{./graphics/}}
\usepackage{enumitem}
\setlist{nosep}
\usepackage{tikz}
\usepackage{tikz-cd}
\usetikzlibrary{matrix,calc,arrows}
\usepackage{hyperref}
\hypersetup{hidelinks}
\newcommand{\FF}{{\mathfrak F}}

\newcommand{\maps}{\colon}

\newcommand{\Ti}{ \mathsf{Tangi}}
\newcommand{\rTi}{ \mathsf{rTangi}}
\newcommand{\fTi}{ \mathsf{fTangi}}
\usepackage{comment}
    \newcommand{\ha}[1]{\textcolor{black}{#1}}
\usepackage{ MnSymbol }
\newcommand{\bangup}{\mathord{!}}              % 0 -> 1
\newcommand{\End}{\mathrm{End}}

\newcommand{\CC}{{\mathcal C}}
\newcommand{\N}{{\mathbb N}} 
\newcommand{\C}{{\mathbb C}}

\newcommand{\id}{{\mathrm{id}}}

\newcommand{\TL}{TL}

\newcommand{\slidy}{slidealisation}
\newcommand{\hal}{${\frac{1}{2}}$-monoidal}
\newcommand{\ham}{${\frac{1}{2}}$-monoidal category} 
\newcommand{\hams}{${\frac{1}{2}}$-monoidal categories}

\newcommand{\ihams}{slideable ${\frac{1}{2}}$-monoidal categories}
\usepackage{comment}
\newcommand{\half}{\frac{1}{2}}
\DeclareMathOperator{\Hom}{Hom}

\newtheorem{theorem}{Theorem}[section]
\newtheorem{proposition}[theorem]{Proposition}
\newtheorem{lemma}[theorem]{Lemma}

\theoremstyle{definition}
\newtheorem{definition}[theorem]{Definition}
\newtheorem{remark}[theorem]{Remark}
\newtheorem{example}[theorem]{Example}
\newtheorem{examples}[theorem]{Examples}
\newtheorem{notation}[theorem]{Notation}
\usepackage{authblk}
\title{Structure and Representations of a Diagrammatic Non-Commutative Tangloid Algebra}

\author[1]{Hadeel B. Albeladi}
\author[2]{Sofia Lambropoulou}
\affil[1]{Department of Mathematics, Faculty of Sciences \& Arts, King Abdulaziz University, Rabigh, Saudi Arabia, E-mail: aealbeladi@kau.edu.sa}
\affil[2]{School of Applied Mathematical and Physical Sciences, National Technical University of Athens, Greece, E-mail: sofia@math.ntua.gr}

\date{\today}

\begin{document}
\maketitle
\begin{abstract}
We introduce the \emph{tangloid algebra} $\mathcal{T}_n$, a diagrammatic algebra arising from the extended tangloid category $\Ti$, and its subalgebra, the \emph{braidoid algebra} $\mathcal{B}_n$. 

The construction begins with the unoriented tangloid category $UTC$, inspired by Turaev's theory of knotoids, which is obtained from the unoriented welded tangleoid category $UWTC$ by removing the welded relation while retaining the forbidden moves and imposing some additional relations. Within this framework, the  braidoid category arises naturally as a subcategory of $UTC$. We then introduce the extended tangloid category $\Ti$ by adjoining a placeholder morphism $\emptyset$, which allows the generators to be extended to endomorphisms of a fixed $n$-box. The interaction of the placeholder morphism with the crossing morphisms gives rise to certain diagonal morphisms, $/$ and $\backslash$, which are conveniently used in the  defining relations of $\Ti$. The corresponding extended braidoid category $\mathbf{Brd}_{\Ti}$ is realised as a subcategory of $\Ti$. 

For each $n\geq 0$, we then define the tangloid algebra $\mathcal{T}_n$ by linearising the endomorphism algebra $\operatorname{End}_{\Ti}(n)$, as a unital $\mathbb{C}$--algebra presented by diagrammatic generators and relations. We also define the braidoid subalgebra $\mathcal{B}_n$, which provides an algebraic framework for the theory of braidoids. We further introduce the reduced tangloid algebra $\delta\mathcal{T}_n$ by imposing two additional reduction relations for trivial knot and trivial knotoid components, and we construct a natural diagrammatic \textit{bilinear pairing} on $\delta\mathcal{T}_n$ using diagrammatic reflection, composition, and closure.

Our categorical and algebraic setting provides a theoretical framework for  studying diagrammatic structures with intrinsic endpoints, namely knotoids, linkoids and braidoids,  and related diagramm algebras. 
\end{abstract}


\section{Introduction}\label{sec1}

Diagrammatic categories arising from knots, links, and tangles play a fundamental role in low-dimensional topology, representation theory,  quantum algebra; see for example \cite{reshetikhin1990ribbon, turaev1990operator, baez1994knots, yetter2001functorial, Ohtsuki2002, kauffman2013knots}.
%buck2009applications, banagl2011mathematics,kauffman1995knot,
Classical diagram algebras such as the Temperley--Lieb algebra \cite{temperley1971percolation} and the BMW algebra \cite{murakami1987kauffman, birman1989braids} provide powerful algebraic and categorical frameworks for the study of braid groups, knot invariants, the Jones polynomial, quantum groups, and topological quantum field theories; see for example  \cite{kassel1995quantum, leduc1997ribbon, jones1985polynomial}. 
% morton2010basis
These structures are typically realized through diagrammatic generators and local isotopy relations inside monoidal or braided tensor categories \cite{reshetikhin1990ribbon, kassel1995quantum, turaev2010quantum, Ohtsuki2002}. 
%kock2004frobenius

More recently, categorical approaches have been extended to open-ended diagrammatic objects, the knotoids \cite{albeladi_2022, gugumcu2021quantum, moltmaker2022framed}. 
Turaev introduced the theory of knotoids \cite{turaev2012knotoids}  as immersions of the unit interval in an oriented surface, up to the classical Reidemeister moves (relations $[T_{1}]-[T_{5}]$ in Figure~\ref{fig:generators_UTC}) and respecting forbidden endpoint moves, thereby extending classical knot theory to diagrams with free endpoints. An example of a knotoid is illustrated in the left-hand side of Figure~\ref{fig:example_knotoid_UTC}.  
Subsequently, the first author introduced   welded tangle-oid categories, generalizing  classical tangle categories to strict monoidal categories including endpoints as well as welded crossings \cite{albeladi_2022}, while G\"ug\"umc\"u and Kauffman \cite{gugumcu2021quantum} and independently Moltmaker \cite{moltmaker2022framed}, developed a categorical framework for Morse multi-knotoids, whereby knot, link, and knotoid diagrams are interpreted as morphisms in a braided monoidal category generated by crossings, cups, caps and endpoint morphisms. 
Their constructions provide a diagrammatic setting for state-sum models and quantum invariants of multi-knotoids and framed multi-knotoids, extending the Reshetikhin--Turaev formalism to knotoid-type objects. 
In all constructions, the endpoints of a welded tangle-oid/Morse knotoid are encoded categorically by morphisms
\[
\upfilledspoon \colon [0]\to[1],
\qquad
\downfilledspoon \colon [1]\to[0],
\]
so that the resulting category may be viewed as an extension of the classical tangle category (cf. for example \cite{turaev1990operator, Ohtsuki2002}) by allowing open strands with free endpoints (see the middle illustration of Figure~\ref{fig:example_knotoid_UTC} for an example of a tangle-oid or tangloid). 

\begin{figure}[H]
    \centering
    \includegraphics[width=0.65\linewidth]{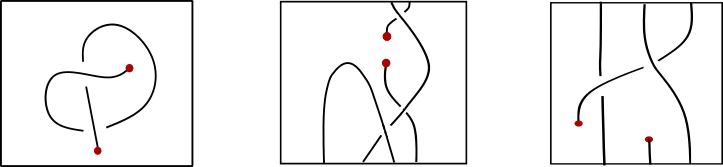}
    \caption{Examples of  a knotoid, a $(2,4)$-tangloid and a braidoid.}
    \label{fig:example_knotoid_UTC}
\end{figure}

The present work is related in spirit to these categorical approaches to knotoids but independent, and differs from them in several important aspects.
 We first define in Section~\ref{subsec:utc} the \textit{unoriented tangloid category}  $UTC$ in terms of generating morphisms and relations among them (see Figures~\ref{fig:UTCg} and \ref{fig:generators_UTC}), taking the approach of sesquicategories \cite{hazewinkel1996handbook} for constructing a presentation of a strict monoidal category.  In particular, $UTC$ is formulated as a strict monoidal category rather than  a braided monoidal framework, so associativity and unit laws hold strictly, without the need for associativity isomorphisms or coherence maps. Furthermore, the category incorporates nontrivial endpoint interaction relations, where endpoint morphisms interact diagrammatically with cups and caps through local relations corresponding to  local isotopies. 
 
 The category $UTC$ is constructed similar to the welded tangle-oid category $UWTC$ in \cite{albeladi_2022}, which extends the classical tangle category  by endpoint morphisms and by welded crossings. However, the category $UTC$ omits the welded crossing structure and respects both forbidden moves for endpoints, whereby strands are not allowed
to pass over or under the endpoints, as supposed to $UWTC$ where only one forbidden move is respected. Moreover, some local isotopy moves reflecting horizontal shifts are included in $UTC$ (relations $[T_{8}]-[T_{11}]$ in Figure~\ref{fig:generators_UTC}). In addition, we define the regular tangloid category $rUTC$, by omitting the first Reidemeister moves (relations $[T_{2}]$ and $[T_{3}]$ in Figure~\ref{fig:generators_UTC}), and the framed  tangloid category $fUTC$, by substituting the first Reidemeister moves by their framed analogues (see Figure~\ref{fig:framed R1}). 
 
 Unlike the Morse knotoid braided monoidal  categories in \cite{gugumcu2021quantum} and  \cite{moltmaker2022framed}, whose primary role is the construction of quantum invariants for knotoids/framed knotoids, the  tangloid category is developed as a diagrammatic monoidal framework suitable for algebraic operations, categorical presentations, and representation-theoretic constructions.  G\"ug\"umc\"u-Kauffman \cite{gugumcu2021quantum} formulated the Morse category  of planar knotoids under regular isotopy moves.  In the Morse category they define the  rotation number for the  oriented open-ended state component. In addition, Moltmaker’s framed and biframed spherical knotoid categories \cite{moltmaker2022framed}    introduce framing and coframing structures related to Reshetikhin--Turaev invariants. The main conceptual difference of the tangloid category from the above works is that in $UTC$ the endpoints are allowed to swing in their diagrammatic regions (relations $[T_{12}]$ and $[T_{13}]$ in Figure~\ref{fig:generators_UTC}). 

\smallbreak
Departing from $UTC$, in this paper we introduce an
\textit{extended unoriented tangloid category}, denoted by $\Ti$. See  Section~\ref{sec:eutc}. 
To construct this category, we first extend $UTC$ by adjoining the
auxiliary empty morphisms $\emptyset_t$ and $\emptyset_b$, which act as
place-keepers at the top and bottom of a diagram respectively, and are
represented diagrammatically by points. These auxiliary morphisms enable
the morphisms of $UTC$ to be extended to morphisms from object $n$ to
object $n$, for any fixed $n$. We then obtain $\Ti$ by adjoining the
empty generator $\emptyset$ to the generating set $E(\beta)$, see Figure~\ref{fig:EUTC_g}. The purpose of this extension
is to provide a fixed-$n$ diagrammatic framework in which the relevant
diagrams can be interpreted as endomorphisms. The presence of the empty and auxiliary morphisms comes with natural diagrammatic composition and tensoring rules detailed in Subsection~\ref{sec:rulesempty}, such as that any arc that becomes incident with a place-keeper point is contracted. The addition of the  empty morphism  leads to extra relations in $\Ti$  beyond those defining $UTC$, including, for example, $[T_4]-[T_5]^\prime$ in
Figure~\ref{fig:placeholder}. 

The composition of a crossing with the placeholder morphism $\emptyset$
gives rise to two special elements of $\Ti$, the \textit{diagonal
morphisms} acting from object $2$ to object $2$, denoted by
$\backslash$ and $/$ (see Figure~\ref{fig:diagonal_morphisms}).
These diagonal morphisms are not additional generators; rather, they are
derived morphisms whose notation is introduced to simplify the expression of the defining relations. 
 The construction and presentation of $\Ti$
are developed in detail in Section~\ref{sec:eutc}, where we also introduce
its regular and framed analogues, $\rTi$ and $\fTi$. 

The construction of $\Ti$ enables us now to define the associated \emph{tangloid algebra}
\[
\mathcal{T}_n=\mathbb C\,\End_{\Ti}(n).
\]
This algebra is viewed as a diagram algebra generated by the tangloid diagrams of degree $n$ modulo the $\Ti$ local relations. We also introduce
the regular and framed analogues, $r\mathcal{T}_n$ and~$f\mathcal{T}_n$.  
\smallbreak

 In this paper we further  introduce the \emph{reduced tangloid algebra}, denoted by $\delta\mathcal{T}_n$, obtained from the tangloid algebra  $\mathcal{T}_n$ by imposing two additional relations under which a trivial knot and a trivial knotoid are each identified with the scalar $\delta \in \C$.  We then construct a \textit{bilinear pairing} on the tangloid algebra $\delta\mathcal{T}_n$ using diagrammatic reflection, composition, and closure, which plays a fundamental role in the algebraic theory developed in this paper. See Section~\ref{ss:bilinear pairing}.

\smallbreak
The theory of braidoids, introduced in
\cite{Gugumcu2017}, cf. \cite{gugumcu2017knotoids,gugumcu2021braidoids, GugumcuKauffmanLambropoulou}, extends classical braid theory to 
the setting of knotoids. It establishes a diagrammatic framework that
captures the topology of knotoids while preserving many of the algebraic
features that make braid theory a powerful tool in knot theory.   A braidoid is a braid-like diagram  with classical strands and two special strands emanating from the two endpoints. 

In this paper we proceed with defining the  \textit{braidoid category} $\mathbf{Brd}$ and the \textit{extended braidoid category} $\mathbf{Brd}_{\Ti}$ as natural subcategories of  $UTC$ and $\Ti$ respectively, where the generators $\cup$ and $\cap$ are omitted.  We then define the \textit{braidoid algebra} $\mathcal{B}_n$ as a subalgebra of the
tangloid algebra $\mathcal{T}_n$. The braidoid algebra $\mathcal{B}_n$ provides an algebraic
structure on the set of braidoids, by means of generators and a full set of relations, sought in
\cite{gugumcu2017knotoids,Gugumcu2017}, thus filling a gap in the literature. In \cite{gugumcu2017knotoids,Gugumcu2017} the diagonal elements are included as building blocks, while here these are derived elements. Also, a set of relations is found but not 
shown to be a complete set. The above are presented in Section~\ref{sec:subcategory_braidoids} and Section~\ref{ss:subalgebra_of_braidoids}. 

A related extension of classical braid theory is the \textit{inverse braid monoid}
$IB_n$, introduced by Easdown and Lavers \cite{EasdownLavers2004}, whose
elements may be viewed as partial braids obtained by deleting some strands
of an ordinary braid. 
The extended braidoid category framework may then be viewed as an endpoint-enhanced analogue of the partial-braid setting underlying the inverse braid monoid, that encompasses the endpoint generators which  retain the open-strand information diagrammatically. At the algebraic level, this suggests a natural extension of the monoid algebra $\mathbb{C}[IB_n]$ by the braidoid algebra $\mathcal{B}_n$. This is a subject of further investigation.

\smallbreak
We view our tangloid categories as simultaneously extending  the classical tangle category and providing  new categorical models for the theory of knotoids, multi-knotoids, which are immersions of the unit interval and a number of circles in an oriented surface, and linkoids, immersions of some copies of the unit interval  in an oriented surface, as well as for braidoids. Recent applications of tangloid structures in mathematical physics and biosystems further motivate this point of view; see, for example, \cite{ozel2024applications, ozel2024biosystems}. 
 The  algebraic structures introduced in this work share features with classical diagram algebras such as the Temperley--Lieb and the BMW algebras \cite{birman1989braids, murakami1987kauffman}, as well as with the Motzkin algebra \cite{Benkart2011Motzkin}, and the inverse braid monoid, while simultaneously incorporating endpoint-type structures analogous to those appearing in knotoid theory. We further explore these threads in sequel work. 
 
\smallbreak
The paper is organized as follows. 
In Section~\ref{sec:utc}, we define the Category $UTC$ and its local relations. 
In Section~\ref{sec:eutc}, we introduce the Extended Unoriented Tangloid Category $\Ti$. In Section~\ref{sec:subcategory_braidoids}, we explain how can the notion of braidoid can be interpreted inside $UTC$ and $\Ti$ by means of the subcategories $\mathbf{Brd}$ and $\mathbf{Brd}_{\Ti}$.
Section~\ref{sec:tangloid-algebra} is devoted to the construction of the tangloid algebra $\mathcal{T}_n$ and its diagrammatic presentation. In Section~\ref{ss:subalgebra_of_braidoids}, we introduce the subalgebra $\mathcal{B}_n$ of braidoids.  Finally, in Section~\ref{sec:pairing-matrix}, we define a standard diagrammatic pairing on the reduced tangloid algebra $\delta\mathcal{T}_n$.

%%%%%%%%%%%%%%%%%%%%%%%%%%%%%%%%%%%%%%%%%%%%%%%%%%%%%%%%%%%

\section{An Unoriented Tangloid Category}\label{sec:utc}

The aim of this section is to define an unoriented tangloid category, denoted $UTC$, following the first author’s construction of the unoriented welded tangloid category $UWTC$ in her thesis \cite{albeladi_2022}.  Roughly, $UTC$ is obtained from  $UWTC$  by removing the welded (virtual) crossing generator $X$ together with all respective relations in which it appears. There are, however, some crucial differences that prevent $UTC$ from being a subcategory of $UWTC$. For once, in $UTC$  both forbidden moves are respected: strands are not allowed
to pass over or under the endpoint maps $!\maps 1\to 0$ and $\text{¡}\maps 0\to 1$, while in $UWTC$ one of the endpoint forbidden moves was allowed. In both  $UWTC$ and $UTC$ we have vertical shift coming from the interchange law, but in $UTC$ we have also horizontal shifts of generators (relations $[T_8]-[T_{11}]$), which are in natural alignment with the topological counterpart category of tangloids.  Moreover,   $UTC$   uses the vertical Reidemeister moves \(R_1\), while \(UWTC\) is presented using the horizontal version of this move. These two forms are topologically analogous since the existence of one yields the other, so both categories describe the same local twisting phenomenon from different geometric viewpoints.

\smallbreak
We begin with a brief discussion of a 
\ham,\, monoidal graphs, and the associated free category, and then we  introduce our definition of the unoriented  tangloid category. 

\subsection{\ham}
%We recall the basic notions needed for the formal presentation of welded tangloids and their linearization. Standard references include \cite{albeladi_2022, higgins1971categories, kassel1995quantum}. The notion of a $\tfrac{1}{2}$--monoidal (``slideable'') structure we use follows the premonoidal framework of Power and Robinson~\cite{power1997premonoidal}.

\begin{definition}[Pre-\hal\, structure](See \cite{power1997premonoidal}).
Let \[\CC=(ob(\CC),\hom_\CC(\_\,,\_),\star ,\id_\_)\] be a category. A \emph{pre-$\half$ -monoidal structure} \[(\CC,I,\otimes_0,  \Theta_{(-)}, {}_{(-)}\Theta)\] in $\CC$ is given by:

\begin{enumerate}
    \item For each pair of objects $x$ and $y$ in $ob(\CC)$, another object $x\otimes_0 y$ in $ob(\CC)$.
    \item An object $I\in ob(\CC)$.
    \item For each morphism $f\colon x \to y$ and object $z$,
    a morphism
    $x\otimes_0 z \xrightarrow{\Theta z(f)} y \otimes_0 z.$
    We will use the notation
   \[\ha{\Big(}x\otimes_0 z \xrightarrow{\Theta z(f)} y \otimes_0 z\ha{\Big)} = \ha{\Big(}x\otimes_0 z \xrightarrow{f \Theta z} y \otimes_0 z\ha{\Big)}.\] 
    \item  For each morphism $f\colon x \to y$ and object $z$, a
    morphism
    $z\otimes_0 x \xrightarrow{z\Theta (f)} z \otimes_0 y.$
     We will use the notation
       \[\ha{\Big(}z\otimes_0 x \xrightarrow{z\Theta (f)} z \otimes_0 y\ha{\Big)} =\ha{\Big(} z\otimes_0 x \xrightarrow{z \Theta f} z \otimes_0 y\ha{\Big)}.\] 
\end{enumerate}

\end{definition}

\begin{definition}[\ham](See \cite{power1997premonoidal}).\label{def:1/2mc}
Let $\CC$ be a category exactly as in the above definition. 
A pre-\hal\, structure
 {on $\CC$}
 \[
 (\CC,I,\otimes_0,  \Theta_{(-)}, {}_{(-)}\Theta)
 \]
 {gives}
 a \emph{\ham}
\[
(\CC,I,\otimes_0,  \#_{(-)}, {}_{(-)}\#)
\]
if the following are satisfied:
\smallbreak

\begin{enumerate}
    \item $(ob(\CC),\otimes_0,I)$ is a monoid.
    \item Let $A$ be an object of $\CC$. Then the pair of assignments
    \begin{itemize} \item
    $ob(\CC) \to ob(\CC)$ such that $ x \mapsto x\otimes_0 A$,
    \item given objects $x,y$, consider \[f \in \hom_\CC(x,y) \mapsto \Theta_A(f) \colon x\otimes_0 A  \to y\otimes_0 A\] 
    \end{itemize}
    a functor $\CC \to \CC$.
    We will denote this functor by $\#_A\colon \CC \to \CC.$ 
    
    \item  Let $A$ be an object of $\CC$. Then the pair of assignments
    \begin{itemize} \item
    $ob(\CC) \to ob(\CC)$ such that $ x \mapsto A\otimes_0 x$.
    \item Given objects $x,y$, consider
     \[f \in \hom_\CC(x,y) \mapsto {}_A\Theta(f) \colon A\otimes_0 x  \to A\otimes_0 y \] 
       a functor $\CC \to \CC$ 
    We will denote this functor by ${}_A\#\colon \CC \to \CC.$
     \end{itemize}
    \item For each $A,B\in ob(\CC)$ 
    \begin{align}
      {}_A\# \circ {}_B\#&={}_{A\otimes_0 B}\#\label{axioms1/2mc1}\\
      \#_A \circ \#_B&=\#_{B\otimes_0 A}\label{axioms1/2mc2}
      \\\#_{A} \circ {}_{B}\#&={}_{B} \#\circ   \#_A\label{axioms1/2mc3}\\
       \#_{I}&=\id_{\CC} \label{axioms1/2mc4}\\ 
       {}_I\# &=\id_{\CC}\label{axioms1/2mc5}.
       \end{align}
       Here $\id_{\CC}$ is the identity functor $\CC \to \CC.$
\end{enumerate}
\end{definition}
\begin{definition}\label{def:slideable}
\ha{Let $\CC=(ob(\CC),\hom_\CC(\_\, , \_),\star, \id)$ be a category}. A \ham\, 
$(\CC, \otimes_0, I, \#_{(-)}, {}_{(-)}\#)$
is called \emph{{slideable}} if given objects $x,y,z,w$ and  a pair of morphisms $f\colon x \to y$ and $g\colon z \to w$ we have
\[ 
(f \Theta w)\star(x \Theta g)
=(y \Theta g) \star (f\Theta z)
\] 
where $f\Theta w \colon x\otimes_0 w \to y\otimes_0 w.$
  This mean that the next diagram commutes
\[
\begin{tikzcd}
x\otimes_0 z \arrow{r}{x\Theta g} \arrow[swap]{d}{f\Theta z}& x\otimes_0 w\arrow{d}{f\Theta w}\\
y\otimes_0 z \arrow[swap]{r}{y\Theta g}& y\otimes_0 w
\end{tikzcd}
\]
\end{definition}
\begin{proposition}\cite{albeladi_2022}
    There is a functor $\FF$ from the category of \hams\, to the category of \ihams.
\end{proposition}
\begin{theorem}\cite{albeladi_2022}
    The category of \ihams\, and the category of strict monoidal categories are ismorphic.
\end{theorem}

%%%%%%%%%%%%%%%%%%%%%%%%%%%%%%%%%%%%%%%%%%%%
\subsection{ Presentation of strict monoidal categories} \label{ss:mcpres}

Let $\beta = (V,E(\beta),\otimes_0,e,\delta_1,\delta_2),$ be a {\it monoidal graph} \cite{albeladi_2022}, where $V$ is the set of vertices, $E(\beta)$ is the set of edges, $\otimes_0$ is the tensor product on vertices, $e$ is the identity element in $V$, and $\delta_1$, $\delta_2$ are the incidence maps where $\delta_1,\delta_2\maps E(\beta)\to V$. The monoidal graph   $\beta$  is a graph  with a monoid structure in the set of vertices. Let $\beta^*$ be the  extent of  $\beta$ which is a monoidal graph that can tensor vertices in $V$ with edges in $E$. Let $P(\beta^*)$ be the free category over $\beta^*$. Then $\Omega(\beta^*)=(P(\beta^*), e, \otimes_0, ({}_a\#)_{a\in V}, (\#_b)_{b\in V}) $ is the {\it free strict monoidal category}, where for all $a,b,k\in V$ \[
{}_a\#_b(k)=a\otimes_0 k\otimes_0 b = a+k+b,
\]
while on a generating arrow $(f \maps k\to k') \in E(\beta) $ it acts as
\[
{}_a\#_b(f):\ a+k+b \xrightarrow{\ a\ \Theta\ f\ \Theta\ b\ } a+k'+b.
\]
Informally, $\Theta$ denotes horizontal juxtaposition.

\smallbreak

To obtain a presentation of a strict monoidal category, one first considers presentations of \hams\,. This is achieved by introducing the notion of a \hal\, congruence and then defining the \hal\, closure $\overline{\underline{W}}$ of a congruence template $W$. The free \ham\ is then quotiented by $\overline{\underline{W}}$. Finally, applying the \slidy\ functor to this quotient yields the desired presentation of the strict monoidal category. We do not recall this construction in detail here; the interested reader is referred to \cite[Chapter~6]{albeladi_2022} for a complete treatment.

%%%%%%%%%%%%%%%%%%%%%%%%%%%%%%%%%%%%%%%%%%%%%%%

\subsection{The unoriented tangloid category $UTC$} \label{subsec:utc}

We now fix the generating data and relations for the tangloid category. Let $V=\N$, and 
\[
E(\beta)=\{X_+,X_-, \cup, \cap, \text{¡}, \bangup\},
\]
with source and target objects determined by the following incidence maps:
\begin{align*}
\delta_1 X_+&=\delta_2 X_+=2,&\ \ \delta_1 X_-&=\delta_2 X_-=2,\\
\delta_1\,\cup&=2,&\delta_2\,\cup&=0,&\delta_1\,\cap&=0,&\delta_2\,\cap&=2,\\
\delta_1\,\text{¡}&=0,&\delta_2\,\text{¡}&=1,&\delta_1\,\bangup&=1,&\delta_2\,\bangup&=0.
\end{align*}
Here $X_\pm$ represent positive/negative (classical) crossings,  $\cup,\cap$ are cup/cap, and $\text{¡}\maps 0\to 1$, $\bangup\maps 1\to 0$ are endpoint maps (head/leg in knotoid terminology). 

\begin{definition}[Unoriented  tangloid category]\label{def:utc}
  Let $\Omega(\beta) =(P(\beta^*), 0, \otimes_0, ({}_a\#)_{a\in \N}, (\#_b)_{b\in \N})$. The \emph{unoriented  tangloid category},  denoted $UTC$, is the strict monoidal category formally presented by
\[
\FF\Big(\Omega(\beta)\big/\, \overline{\underline T}\Big),
\]
where $\overline{\underline T}$ is the  $\tfrac{1}{2}$--monoidal closure of the congruence $\underline{T}$ generated by the local relations $[\mathrm{T}_i]$ below and their left/right insertions (the composition of morphisms follows the composition of functions convention): 

In $\hom_{P(\beta^*)}(0,1)$, we have the only relation:

 \smallbreak
    \begin{itemize}
    \item {$[T_{12}]:\cap \circ (\mathrm{id}_1\otimes !)\sim_{T_{0,1}}\, \text{¡}\sim_{T_{0,1}}\,\cap\circ (!\otimes \mathrm{id}_1) $}
    \end{itemize}
    
\bigbreak
\noindent     In $\hom_{P(\beta^*)}(1,0)$, we have the only relation:

 \smallbreak
    \begin{itemize}
    \item {$ [T_{13}]: (\mathrm{id}_1\otimes \text{¡}) \circ \cup\sim_{T_{1,0}}\,\,  !\sim_{T_{1,0}} (\text{¡} \otimes \mathrm{id}_1)\circ\cup $}
     \end{itemize}
    
\bigbreak
\noindent In $\hom_{P(\beta^*)}(1,1)$, we have the only relations:

 \smallbreak
\begin{itemize}
 \item {$[T_1]:(\cap\otimes \mathrm{id}_1)\circ(\mathrm{id}_1\otimes \cup) \sim_{T_{1,1}} \mathrm{id}_1\sim_{T_{1,1}} (\mathrm{id}_1 \otimes \cap)\circ(\cup\otimes \mathrm{id}_1)$}
    \smallbreak
    \item $[T_{9}]: \text{¡}\, \otimes\, ! \sim_{T_{1,1}} \, ! \circ \text{¡} \sim_{T_{1,1}} !\, \otimes \,\text{¡}$
    \end{itemize}
\bigbreak
\noindent In $\hom_{P(\beta^*)}(0,2)$, we have the only relations:
\begin{itemize}
    \item $[T_2]: (\cap)\circ X_- \sim_{T_{0,2}} \cap \sim_{T_{0,2}}  (\cap ) \circ X_+ $
\end{itemize}
   \bigbreak
\noindent In $\hom_{P(\beta^*)}(2,0)$, we have the only relations:
\begin{itemize}
    \item $[T_3]:(X_-)\circ \cup \sim_{T_{2,0}}  \cup \sim_{T_{2,0}}  (X_+)\circ\cup$
\end{itemize}
 \bigbreak
 
\noindent     In $\hom_{P(\beta^*)}(1,2)$, we have the only relation:
     \begin{itemize}
         \item $[T_{10}]: \cap\, \otimes\, ! \sim_{T_{1,2}}\,  ! \circ \cap \sim_{T_{1,2}}\,  ! \otimes \cap $
     \end{itemize}
       \bigbreak
\noindent     In $\hom_{P(\beta^*)}(2,1)$, we have the only relation:
       \begin{itemize}
           \item $[T_{11}]: \text{¡}\, \otimes \cup \sim_{T_{2,1}}\cup \circ\, \text{¡} \sim_{T_{2,1}} \cup \otimes \text{¡} $
       \end{itemize}
    \bigbreak
\noindent     In $\hom_{P(\beta^*)}(2,2)$, we have the only relations:

 \smallbreak
    \begin{itemize}
     \item {$[T_4]:X_- \circ X_+\sim_{T_{2,2}} \mathrm{id}_2\sim_{T_{2,2}} X_+ \circ X_-$}
    \end{itemize}
    \smallbreak
    \begin{itemize}
        \item $[T_{8}]: \cap \otimes \cup \sim_{T_{2,2}} \cup \circ \cap \sim_{T_{2,2}} \cup \otimes \cap$
    \end{itemize}

 \bigbreak

\noindent    In $\hom_{P(\beta^*)}(1,3)$, we have the only relations:

 \smallbreak
    \begin{itemize}

 \item $[T_{6}]:(\cap\otimes \id_1)\circ (\id_1\otimes X_+)\sim_{T_{1,3}} (\id_1\otimes \cap)\circ (X_-\otimes \id_1)$
     \smallbreak
    \item  {$[T_{7}]:(\cap \otimes \mathrm{id}_1)\circ (\mathrm{id}_1\otimes X_-) \sim_{T_{1,3}} (\mathrm{id}_1\otimes \cap)\circ (X_+\otimes \mathrm{id}_1)$}

    \end{itemize}

    \bigbreak
\noindent   In $\hom_{P(\beta^*)}(3,3)$, we have the only relations:

 \smallbreak
    \begin{itemize}
    \item {$[T_5]: (X_-\otimes \mathrm{id}_1)\circ(\mathrm{id}_1\otimes X_-)\circ (X_-\otimes \mathrm{id}_1)\sim_{T_{3,3}} (\mathrm{id}_1 \otimes X_-)\circ (X_-\otimes \mathrm{id}_1)\circ (\mathrm{id}_1 \otimes X_-)$}
    \end{itemize}
    
\end{definition}

\noindent The generators and relations above can be understood geometrically, as shown in Figures~1 and~2. Each generator is a local diagrammatic piece, and each relation says that two local pictures represent the same morphism whenever they are related by the allowed topological moves. The tensor product corresponds to horizontal juxtaposition of diagrams, whereas composition corresponds to vertical stacking.

Thus, when two compatible arcs meet under composition, their common boundary point is identified and the arcs join to form a continuous strand. We read  diagrams top-to-bottom.  In this way, these relations encode the topological concatenation and isotopy rules underlying the category.
\begin{figure}[H]
    \centering
\includegraphics[width=0.8\linewidth]{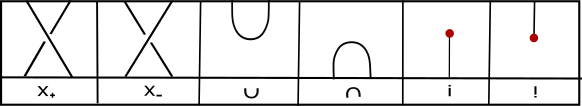}
    \caption{Geometric representations for the generators of $UTC$. }
    \label{fig:UTCg}
\end{figure}

\begin{figure}[H]
    \centering
\includegraphics[width=0.8\linewidth]{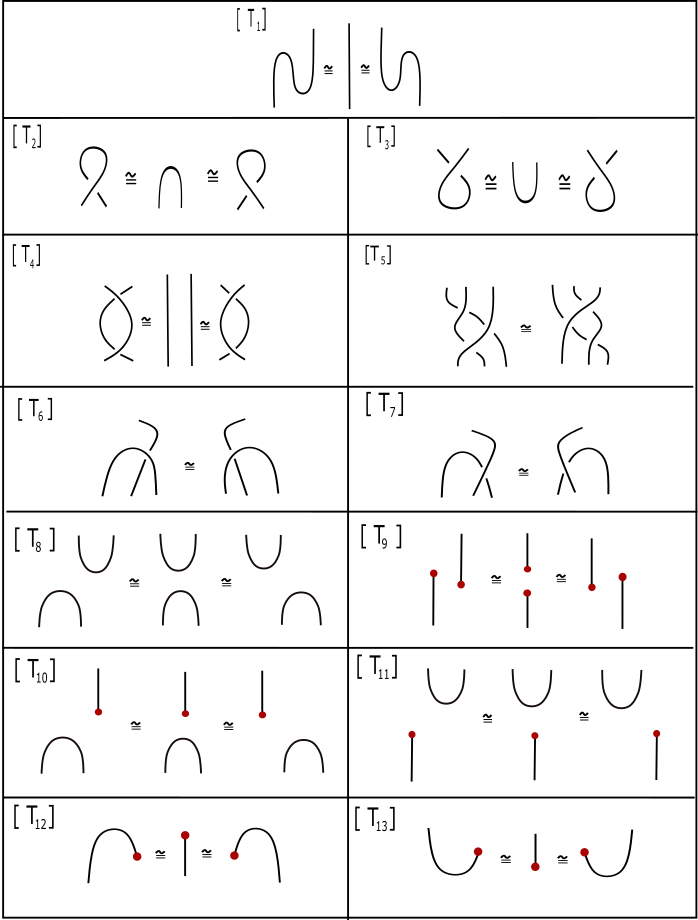}
    \caption{Diagrammatic representations for the relations of $UTC$. }
    \label{fig:generators_UTC}
\end{figure}

Note that the cup swing moves can obtained from the cap swing moves \cite{Ohtsuki2002} as in the figure below,
\[
\includegraphics[width=0.8\linewidth]{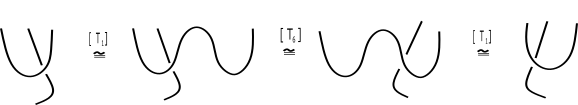}
    \label{fig:cup_swing_moves}
\]

It turns out that adding the swing moves $[T_{6}], [T_{7}]$ to our list of isotopy moves makes
redundant certain instances of Reidemeister moves involving horizontal arcs. For example, one can easily verify that a horizontal RI move can be obtained from a vertical one and two swing moves. Likewise, an RII move with two horizontal arcs can be obtained by an RII
move with two vertical arcs, two swing moves and changes of relative positions of vertices.

 Since $UTC$ is strict monoidal, we freely use the \emph{interchange law}:
for any composable morphisms $f,f',g,g'$ we have
$(f\otimes g)\circ(f'\otimes g')=(f\circ f')\otimes(g\circ g')$.
In particular, morphisms supported on disjoint tensor factors commute.  For example, assuming composable identity morphisms we have:
\begin{align}  \label{general_far_commuting}
    (f\otimes \id) \circ (\id\otimes g)  = (f \circ \id) \otimes (\id \circ  g)  = f\otimes g= (\id \circ f) \otimes (g \circ \id ) = (\id \otimes g) \circ (f \otimes \id)
\end{align}
So, thanks to the interchange law and the identity morphism, we obtain the exchange of horizontal levels in the diagrammatic counterpart. As another instance, we obtain the exchange of horizontal levels of  $\cup$,  $\cap$ or a crossing with an endpoint morphism: 
\begin{align}  \label{eq:crossing_shift}
 (X_{\pm} \otimes  \id)  \circ  (\id_2 \, \otimes \, !) =  (X_{\pm} \circ \id_2) \otimes   (\id \, \circ\, !) = (\id_2 \circ X_{\pm}) \otimes   (\, !\,) =X_{\pm}\, \otimes\, ! 
\end{align}
\begin{align}  \label{far_commuting_endpoint}
 (\cap \otimes  \id)  \circ  (\id_2 \, \otimes \, !) =  (\cap \circ \id_2) \otimes   (\id \, \circ\, !)= (\cap \circ \id_2) \otimes   (\,!\,)= \cap\, \otimes\, ! 
\end{align}
\begin{align}  \label{far_commuting_endpoint22}
(\id_2 \, \otimes \text{¡} \, ) \circ  (\cup \otimes  \id)     =  (\id_2 \circ \cup) \otimes   (\text{¡}\, \circ\, \id)=(\id_2 \circ \cup) \otimes   (\text{¡}\,)  = \cup\, \otimes\, \text{¡}
\end{align}
together with all similar versions of these relations.  Figure~\ref{fig:endpoint_shifts}, top rows, illustrates the above endpoint shifts.  All symmetric ones follow similarly. We do not list these ``far-commuting'' relations separately. 

One can also achieve from the basic relations a horizontal shift between two endpoints: 
\begin{align*}
%\label{endpoints-relation}
    !\,\circ  \cap \circ (\, \text{¡}\, \otimes \id_2 )\circ (\cup \otimes \id )&= \, !\,\circ ( \cap) \circ (\, \text{¡}\, \otimes (\id \otimes \id)  )\circ (\cup \otimes \id ) \\
    &= \, !\,\circ ( \cap) \circ ((\, \text{¡}\, \otimes \id )\otimes \id  )\circ (\cup \otimes \id ) \\
    &= \, !\,\circ ( \cap) \circ  ((\, \text{¡}\, \otimes \id )\circ \cup)\otimes (\id \circ \id)\\
    &\stackrel{[T_9]}{=} \, !\,\circ (\cap \,\circ (\,! \,\otimes  \id)) \stackrel{[T_8]}{=}  \, !\, \circ \,  \text{¡} 
\end{align*}
By similarity, 
\begin{align*}
    ! \, \circ \cap\, \circ ( \id_2 \otimes \text{¡} )\circ (\id \otimes \cup) \stackrel{[T_9,T_8]}{=} \, !\, \circ \,  \text{¡} 
\end{align*}
Hence, we obtain the combined relations  (see bottom row of Figure~\ref{fig:endpoint_shifts}):
\begin{equation}\label{endpoint_shift}
    !\,\circ  \cap \circ (\, \text{¡}\, \otimes \id_2 )\circ (\cup \otimes \id )  =  \, !\, \circ \,  \text{¡} = ! \, \circ \cap\, \circ ( \id_2 \otimes \text{¡} )\circ (\id \otimes \cup)
\end{equation}
%%%%%%%%%%%%%%%%%%%%%%%%%%%%%%%%%%%%%%%%%%%%%%%%%%%%%%%%%%%%%%% 

%%%%%%%%%%%%%%%%%%%%%%%%%%%%%%%%%%%%%%%%%%%%%%%%%%%%%%%%%%%%%%%%

\begin{figure}[H] 
    \centering
\includegraphics[width=0.6\linewidth]{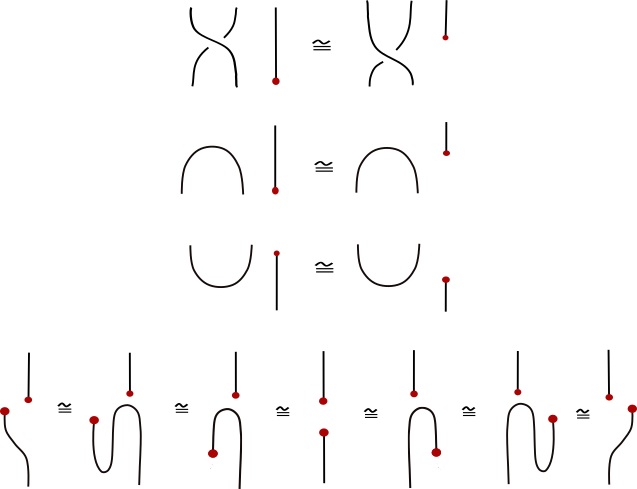}
    \caption{Geometric representations for relations (\ref{eq:crossing_shift}), (\ref{far_commuting_endpoint}), (\ref{far_commuting_endpoint22}) and (\ref{endpoint_shift}). }
    \label{fig:endpoint_shifts}
\end{figure}

\begin{remark}[Why "unoriented"]
Knotoids typically carry an orientation from leg to head. Here we work with unoriented strands to keep the presentation minimal and to isolate the algebraic effect of endpoint generators. An oriented variant can be obtained by duplicating generators and polarities.
\end{remark}

\begin{remark}
    Let $UTC^*$ denote the subcategory of $UWTC$,  obtained  by omitting the generator $X$, representing a welded crossing, together with all relations involving $X$. Recall that in  $UWTC$ one forbidden move is allowed, namely the one in which the endpoint arc  can move under any other other arc.  This move is, likewise, inherited in $UTC^*$. On the other hand,   in $UTC$  both forbidden moves are respected, namely it is forbidden to move an endpoint arc over or under other arcs. Moreover, $UTC$ includes relations $[T_8]-[T_{11}]$ reflecting horizontal and vertical shifts of generating morphisms.   So, there is an obvious surjective functor  from $UTC$ to $UTC^*$.
\end{remark}

\begin{remark}
 There are several differences between the category $UTC$ and the Morse knotoid categories that are defined independently in \cite{gugumcu2021quantum} and in \cite{moltmaker2022framed}, where the primary role is the construction of quantum invariants for planar knotoids and spherical framed knotoids respectively. In particular, the strict monoidal category framework  and the relations of endpoints with cup and cap  $[T_{12}]$ and $[T_{13}]$ give the $UTC$ more flexibilty  as we see in Figure~\ref{fig:endpoint_shifts}. Also these relations $[T_{12}]$ and $[T_{13}]$ are very natural to include, as they reflect allowed surface moves in the theory of knotoids. An example is illustrated in Figure~\ref{fig:cup_cap_relations}. The absence of these relations  in \cite{gugumcu2021quantum,moltmaker2022framed} caused them to entrain the rotation number or the oriented open-ended state component. In addition, Moltmaker’s
framed and biframed knotoid categories introduce framing and coframing structures related to the 
Reshetikhin–Turaev invariants.
\end{remark}
\begin{figure}[H]
    \centering
    \includegraphics[width=0.4\linewidth]{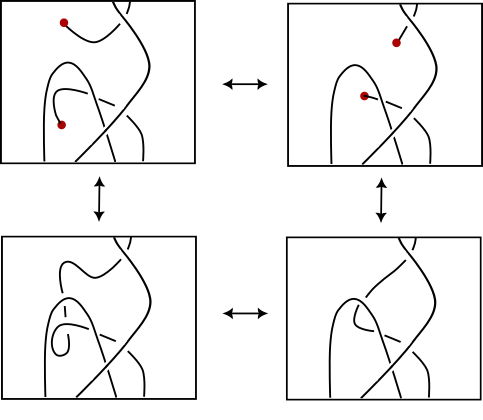}
    \caption{Justifying the relations $[T_{12}]$ and $[T_{13}]$.}
    \label{fig:cup_cap_relations}
\end{figure}

 We close this section by considering the regular and the framed analogues of $UTC$ in parallel to the works 
  \cite{gugumcu2021quantum} and  \cite{moltmaker2022framed}.

\begin{definition} \label{def:regular-framed_UTC}
We define the \emph{regular unoriented tangloid category}, denoted $rUTC$, to have the same generating morphisms as the unoriented tangloid category $UTC$, but with the relations $[T_2]$ and $[T_3]$ omitted. These relations correspond diagrammatically to the Reidemeister 1 moves, thus by omitting them we are in the regular isotopy tangle category.

Furthermore, $UTC$ gives rise to the \emph{framed unoriented tangloid category}, denoted $fUTC$,  having the same generators and relations as $UTC$, except for the relations $[T_2]$ and $[T_3]$ which are replaced by the framing preserving relations depicted in Figure~ \ref{fig:framed R1}. 
\end{definition}

\begin{figure}[H]
    \centering
\includegraphics[width=0.6\linewidth]{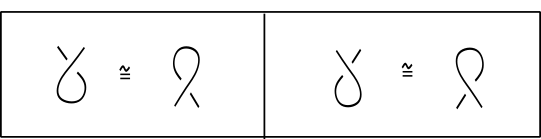}
    \caption{The framing preserving relations in $fUTC$. }
    \label{fig:framed R1}
\end{figure}

Our motivation for Definition~\ref{def:regular-framed_UTC} is that quantum invariants are often not invariant under the
first Reidemeister move, yet regular isotopy invariants can often be normalized to ambient isotopy invariants by a normalizing factor. On the other hand framed isotopy is needed in the construction of quantum invariants for framed knots.

%%%%%%%%%%%%%%%%%%%%%%%%%%%%%%%%%%%%%%%%%%%%%%%%%%%%%%%%%%%%%%%%%%%%

\section{The Extended Unoriented Tangloid Category}\label{sec:eutc}

In this section we define an extended unoriented tangloid category, denoted $\Ti$. This category  is obtained from $UTC$ by adding the empty generator $\emptyset$ to $E(\beta)$. For doing this, we first extend  $UTC$ by auxilliary empty morphisms, $\emptyset_t$ and $\emptyset_b$, acting as place-keepers at the top and bottom, for enabling to have every morphism from object $n$ to object $n$ for any $n$. Diagrammatically, the place-keeper morphisms represent  points.

\subsection{Generating data for $\Ti$ - the empty morphism }

We now fix the generating data  for the tangloid category $\Ti$ (recall Subsection~\ref{ss:mcpres}). Let
\[
E(\beta)=\{X_+,X_-, \cup, \cap,  \, \text{¡}, \bangup, \emptyset\},\]

As in $UTC$, $X_\pm$ represent positive/negative (classical) crossings,  $\cup,\cap$ are cup/cap, and $\text{¡}, \bangup$ are endpoint maps (head/leg in knotoid terminology). 
 In $\Ti$, however, we would like each generator to have the same source and target. For this reason we extend $UTC$ to $UTCe$ by including the auxilliary \textit{top and bottom  empty morphisms} $\emptyset_t \maps 1\to 0$ and $\emptyset_b \maps 0\to 1$. These are represented diagrammatically as a top/bottom black dot $\bullet$ : \includegraphics[]{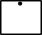} and \includegraphics[]{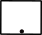}, which act as placeholders for having every generator of  $\Ti$ acting from object $2$ to object $2$, except for $\bangup$ and $\text{¡}$  that now act from object $1$ to object $1$. So, the empty morphisms appear in the representations of all generators, except for the crossings. In other words we define in $\Ti$, 
\begin{align*}
      \cup  \ \maps  \  2\to 2  :=& (\cup \in UTC)  \circ (\emptyset_b \otimes \emptyset_b) \\
      \cap \ \maps  \  2\to 2  :=& (\emptyset_t \otimes \emptyset_t) \circ  (\cap \in UTC)\\
      \text{¡} \ \maps   \ 1\to 1 :=& \emptyset_t  \circ  (\text{¡} \in UTC)\\
      \bangup \ \maps  \  1\to 1  :=& (\bangup \in UTC)  \circ  \emptyset_b
 \end{align*}
 
\noindent  In the set of generators  $E(\beta)$ of $\Ti$  we further add the  \textit{empty morphism} 
 \[
 \emptyset := \emptyset_t \circ \emptyset_b \ \maps \ 1\to 1
 \]
\noindent represented diagrammatically as two corresponding black dots. In Figure~\ref{fig:EUTC_g} the elements of $E(\beta)$ are represented geometrically. 
 Hence, the source and target objects for the generators of $\Ti$ are now extended to the following:  
\begin{alignat*}{2}
\delta_1 X_+&=\delta_2 X_+=2,\\
\delta_1 X_-&=\delta_2 X_-=2,\\
\delta_1\,\cup&=\delta_2\,\cup=2,\\
\delta_1\,\cap&=\delta_2\,\cap=2,\\
%\delta_1\,/&=\delta_2\, / = 2 ,\\
%\delta_1\,\backslash&=\delta_2\,\backslash=2,\\ 
\delta_1\,\text{¡}&=\delta_2\,\text{¡}=1,\\
\delta_1\,\bangup &=\delta_2\,\bangup=1,\\
\delta_1 \, \emptyset&=\delta_2 \,  \emptyset =1.
\end{alignat*} 

\begin{figure}[H]
    \centering
    \includegraphics[width=0.9\linewidth]{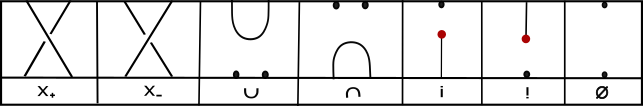}
    \caption{Geometrical presentation of the $\Ti$ generators.}
    \label{fig:EUTC_g}
\end{figure}

 %%%%%%%%%%%%%%%%%%%%%%%%%%%%%%%%%
 
 \subsection{Composition  and  tensoring rules for $\emptyset$ } \label{sec:rulesempty}

  Before providing the definition of $\Ti$  we explain the composition  and  tensoring rules especially for $\emptyset$ in $\Ti$, as they arise from  the general diagrammatic rules below. 
  
 \begin{enumerate}  
 \vspace{.25cm}
\item When two arcs become incident they join and when two  $\bullet$'s  become incident they merge into one, as for example in Figure~\ref{fig:ex2} or Figure~\ref{fig:ex1}, and vice versa.
     \vspace{.25cm}
     
 \item When two morphisms are composed, any  isolated $\bullet$ in an interior region  cancels, as for example in Figure~\ref{fig:ex2}. Conversely, an  isolated $\bullet$ can emerge.
         \vspace{.25cm}
         
 \item When two morphisms are composed, any arc that becomes incident with $\bullet$  is contracted to $\bullet$, whether it is an arc of a cup/cap, crossing,  or it contains an endpoint or not. See for example  Figures~\ref{fig:ex2} and~\ref{fig:ex3}.
     \vspace{.25cm}

 \item For any morphism $L$, the notation $L \otimes \emptyset$ means that we tensor $L$ with $\emptyset$ by placing one $\bullet$ at the bottom right of $L$ and another corresponding $\bullet$ at the top right of $L$. Analogously for the notation $\emptyset\otimes  L$.
 %    \vspace{.25cm}
\end{enumerate}

 \begin{figure}[H]
    \centering
\includegraphics[width=0.5\linewidth]{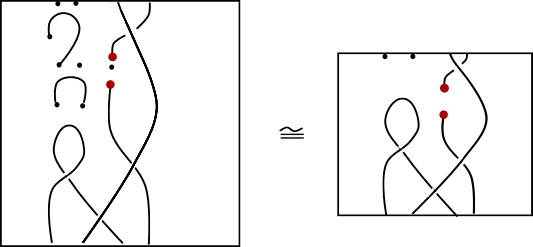}
    \caption{Contracting arcs and points}
    \label{fig:ex2}
\end{figure}

 \begin{figure}[H]
    \centering
\includegraphics[width=0.4\linewidth]{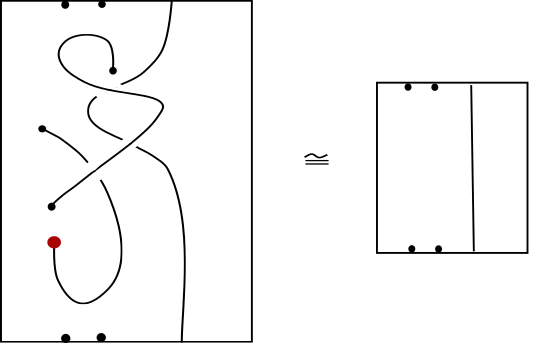}
    \caption{Contracting arcs and endpoints}
    \label{fig:ex3}
\end{figure}

\begin{notation}
As usual, for any $n\in \mathbb{N}, n>1,$ we adopt the notations $\id_n := \id_1 \otimes \ldots \otimes \id_1$ and $\emptyset_n := \emptyset \otimes \ldots \otimes \emptyset$.
\end{notation}
%{\color{red} For any morphism $L$, the notation $L \otimes \emptyset^s$ means that we tensor $L$ with $\emptyset$ by placing one $\emptyset$ at the bottom (source) of $L$. For any morphism $L$, the notation $L \otimes \emptyset^t$ means that we tensor $L$ with $\emptyset$ by placing one $\emptyset$ at the top (target) of $L$.}

 The above diagrammatic rules can be summarized in the following:  
 
 \begin{definition} \label{def:emptyset rules}
     For the $\emptyset$ in $\Ti$ we have the following {\it composition rules} and {\it tensoring rules}. 

  \vspace{.25cm} 
\begin{enumerate}
 \item $\emptyset$ composed with $\emptyset$ is contracted to $\emptyset$: $\emptyset \, \circ \emptyset = \emptyset$. Generally,  $\emptyset_n \, \circ\emptyset_n = \emptyset_n$
    \vspace{.25cm}
    
    \item $\id_1$ composed with $\emptyset$ is contracted to $\emptyset$: $\id_1 \, \circ\emptyset = \emptyset = \emptyset \, \circ \id_1$ 

    Generally,  $\id_n \, \circ\emptyset_n = \emptyset_n = \emptyset_n \, \circ \id_n$
    \vspace{.25cm}
 
      \item  
    \begin{align*}
        (\emptyset  \otimes \emptyset)\circ\,\cup  & = (\emptyset  \otimes \id_1)\,\circ \cup=(\id_1  \otimes \emptyset)\,\circ \cup = \emptyset \otimes \emptyset  \\
        \ \cup \circ\,(\emptyset \otimes \emptyset) \,  & = \cup\,\circ(\id_1\otimes \emptyset)=\cup\,\circ(\emptyset\otimes \id_1)= \cup \\
        \cup \circ \cup &= \cup
    \end{align*}
    
      \item 
      \begin{align*}
          \cap \, \circ (\emptyset  \otimes \emptyset) & =\cap \,  \circ (\emptyset  \otimes \id_1) =\cap \, \circ (\id_1  \otimes \emptyset)= \emptyset \otimes \emptyset\\
          (\emptyset \otimes \emptyset)\, \circ\cap & = (\id_1 \otimes \emptyset)\, \circ\cap =(\emptyset \otimes \id_1)\,\circ \cap= \cap\\
           \cap \circ \cap &= \cap
      \end{align*}
      
     \item  \begin{align*}
         X_{\pm} \, \circ (\emptyset  \otimes \emptyset) & = \emptyset \otimes \emptyset  = (\emptyset \otimes \emptyset) \circ \, X_{\pm}\\
        (\id_1  \otimes \emptyset)\,\circ X_{+} & = (\id_1  \otimes \emptyset)\,\circ X_{-}  \\ 
         (\id_1  \otimes \emptyset)\,\circ X_{+} & = X_{\pm} \, \circ (\emptyset  \otimes \id_1)    \\
(\emptyset \otimes \id_1)\circ \, X_{+} & = (\emptyset \otimes \id_1)\circ \, X_{-} \\ 
  (\emptyset \otimes \id_1)\circ \, X_{+} & =        
         X_{\pm} \circ(\id_1 \otimes \emptyset) 
     \end{align*}

     \item  \begin{align*}
         \text{¡}\circ \emptyset&= \emptyset\\
         \emptyset \circ\text{¡} &=\text{¡}
     \end{align*}
     
     \item \begin{align*}
         !\circ \emptyset &=\,\, !\\
         \emptyset\,\circ\, ! &= \, \emptyset
     \end{align*}
     
     \item \begin{align*}
   ( \emptyset \otimes  \text{¡} ) \circ X_+ & =  ( \emptyset \otimes  \text{¡} )\circ  X_-\\
   ( \text{¡} \otimes  \emptyset ) \circ X_+ & = ( \text{¡} \otimes  \emptyset )\circ  X_-\\
   X_+ \circ  ( \emptyset \otimes  \text{¡} ) & = (\emptyset\otimes \text{¡}) =  X_- \circ ( \emptyset \otimes  \text{¡} ) \\
  X_+ \circ ( \text{¡} \otimes  \emptyset )& = \text{¡} \otimes  \emptyset =  X_- \circ  ( \text{¡} \otimes  \emptyset )
     \end{align*}
     
    \item  \begin{align*}
   ( \emptyset \, \otimes\,  ! ) \circ X_+ &= (\emptyset \, \otimes\,!) = ( \emptyset\,  \otimes  ! )\circ  X_-\\
   ( !\otimes  \emptyset ) \circ X_+&= (! \otimes \emptyset) \,  = ( ! \otimes  \emptyset ) \circ X_-\\
   X_+ \circ ( \emptyset\,  \otimes\, ! ) & =  X_- \circ ( \emptyset\,  \otimes\,  ! ) \\
  X_+ \circ (! \otimes  \emptyset )& =  X_- \circ ( !\otimes  \emptyset )
     \end{align*}

     \item  \begin{align*}
        (\text{¡}  \otimes \emptyset)\,\circ \cup  & = \emptyset \otimes \emptyset =( \emptyset \otimes  \text{¡} )\, \circ \cup  \\
        \cup \,\circ  (!\otimes \emptyset) & = \cup = \cup \, \circ (\emptyset\,\otimes !)\\
         (!\otimes \emptyset) \circ  \cup \, & =\, ! \otimes \emptyset \\
           (\emptyset\,\otimes\, !) \circ\cup &= \emptyset\, \otimes\, !
    \end{align*}
    
    \item  \begin{align*}
       \cap \,\circ (\text{¡}  \otimes \emptyset)\,  & =  \text{¡}\otimes \emptyset \\
       \cap \,\circ  (\emptyset \otimes  \text{¡} ) &= \emptyset\, \otimes \text{¡}  \\
        (\text{¡}  \otimes \emptyset)\,\circ \cap  & = \cap =( \emptyset \otimes  \text{¡} )\, \circ\cap \\
        \cap \,\circ (!  \otimes \emptyset)\,  & = \emptyset \otimes \emptyset =\cap \circ\, (\emptyset \otimes  ! ) 
    \end{align*}

\end{enumerate}
\end{definition}

%%%%%%%%%%%%%%%%%%%%%%%%%%%%%%%%%%%%%%%%%%%%%%%%
\subsection{The diagonal elements}

 The composition of a crossing with the placeholder morphism $\emptyset$ gives rise to two special elements in $\Ti$, the \textit{diagonal morphisms} acting from object $2$ to object $2$, denoted by  $\backslash$ and $/$, as illustrated in Figure~\ref{fig:diagonal_morphisms}:
 
\[
\backslash := (\id_1 \otimes \emptyset)\circ X_+ \quad \text{and} \quad  
/  := (\emptyset \otimes \id_1)\circ X_+ 
\]

\noindent These diagonal morphisms will subsequently appear as shorthand in the defining relations involving the generators. 
 
 \begin{figure}[H]
     \centering
     \includegraphics[width=0.6\linewidth]{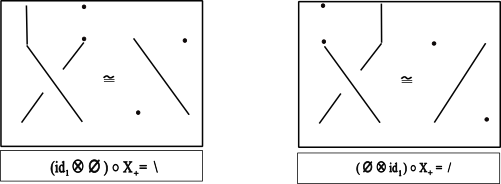}
     \caption{The diagonal morphisms $\backslash$ and $/$.}
     \label{fig:diagonal_morphisms}
 \end{figure}

 Note that, in the extension $UTCe$ of $UTC$ the diagonal elements can be obtained by tensoring the identity with $\emptyset_t$ or $\emptyset_b$ composed with cup-cap. For example,
 \[
 \backslash = (\id_1\otimes \emptyset_t)\circ (\id_1\otimes\cap)\circ (\cup \otimes \id_1)\circ (\emptyset_b\otimes \id_1)
 \]

For the diagonal elements we have the composition and tensoring rules below,  which follow from the  ones listed in Definition~\ref{def:emptyset rules}:

 \begin{enumerate}
      \item  \begin{align*}
         (/)\, \circ (\emptyset \otimes\emptyset) & = \emptyset \otimes\emptyset =  (\emptyset \otimes\emptyset)\circ \,(/)  \\
         (/)\,\circ(\emptyset \otimes \id_1)&= \emptyset \otimes\emptyset = (\id_1 \otimes \emptyset)\circ \, (/) \\
         (/)\,\circ  (\id_1 \otimes \emptyset)   & = / =(\emptyset \otimes \id_1)\circ \, (/ )
     \end{align*}
     
     \item   \begin{align*}
         (\backslash)\, \circ(\emptyset \otimes\emptyset) & = \emptyset \otimes\emptyset =(\emptyset \otimes\emptyset) \circ\,(\backslash) \\
         (\emptyset \otimes \id_1)\circ \,(\backslash) &= \emptyset \otimes\emptyset=  (\backslash)\circ (\id_1 \otimes \emptyset)\,\\
          ( \id_1 \otimes \emptyset)\,\circ (\backslash) & = \backslash = (\backslash)\circ \, (\emptyset \otimes \id_1) \,\, 
     \end{align*}

\item
\begin{align*}
     (\emptyset \otimes \text{¡}) \,\circ   (/) &=  ( \emptyset \otimes  \text{¡} ) \circ X_+   \\
   (\text{¡} \otimes \emptyset)\circ  \, (\backslash)  &= ( \text{¡} \otimes  \emptyset ) \circ X_+ 
\end{align*}
\item 
\begin{align*}
  (\backslash)\circ (\emptyset\, \otimes\, !)    &=X_+ \circ ( \emptyset\,  \otimes\, ! )  \\
 (/)\circ  (!\, \otimes  \emptyset) &= X_+ \circ (! \otimes  \emptyset )
\end{align*}

   \item \begin{align*}
 ( \text{¡} \otimes  \emptyset ) \circ\, (/) & = \emptyset\otimes \emptyset=   ( \emptyset \otimes  \text{¡} )\, \circ(\backslash) \\
 (/) \, \circ ( \text{¡} \otimes  \emptyset ) & = \text{¡}\otimes\emptyset = (\backslash)\,\circ ( \text{¡} \otimes  \emptyset )\\
  (/) \circ(\emptyset \otimes \text{¡})&= \emptyset \otimes \text{¡}= (\backslash)\,\circ (\emptyset \otimes \text{¡})
     \end{align*}
     
\item \begin{align*}
  (/)\circ(\emptyset\, \otimes\,!) & = \emptyset\, \otimes \emptyset = (\backslash)\circ(!\otimes \emptyset)\\
  (\emptyset\, \otimes\,!)\, \circ(/) & = \emptyset \,\otimes\, != (\emptyset\,\otimes\,!)\, \circ(\backslash)\\
  (!\otimes \emptyset)\circ(\backslash) & =\, !\,\otimes \emptyset= (!\otimes \emptyset)\circ(/)
\end{align*}

\item \begin{align*}
    \backslash \circ X_+ &= \backslash \circ / = \backslash \circ X_- \\
     / \circ X_+ &= / \circ \backslash = / \circ X_- \\
    X_+ \circ \backslash &= / \circ \backslash = X_- \circ \backslash\\
    X_+ \circ / &= \backslash \circ / = X_- \circ / 
\end{align*}
 \end{enumerate}

\begin{figure}[H]
    \centering
    \includegraphics[width=0.8\linewidth]{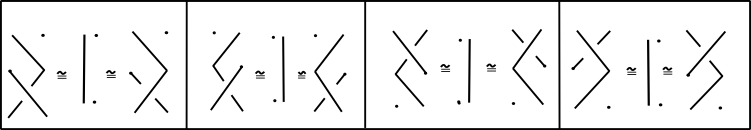}
    \caption{Composing diagonal with crossing makes an arc tensored with $\emptyset$.}
    \label{fig:diagonal_and_crossing}
\end{figure}
 \vspace{.25cm}

 %%%%%%%%%%%%%%%%%%%%%%%%%%%%%%%%%%%%%%%%%%%%%%%%

\begin{examples}
Here  are some examples of morphisms in $\Ti$. 
 \vspace{.25cm}
   
\begin{enumerate}
    \item $ (\cap \otimes X_+)\circ( / \otimes\, ! \otimes \id_1)\circ ( \cap \otimes \text{¡} \otimes \id_1)\circ( \cap \otimes X_- )\circ( X_- \otimes \id_2) \circ(\id_1 \otimes X_- \otimes \id_1) $, see Figure~\ref{fig:ex2}.
     \vspace{.25cm}
   
    \item  $ (\cap \otimes \id_1) \circ(\backslash \otimes \id_1) \circ( \id_1 \otimes X_+ )\circ(\id_1\otimes X_-)\circ(X_-\otimes \id_1)\circ( \text{¡} \otimes \id_2) \circ( \cup \otimes \id_1)$, see Figure~\ref{fig:ex3}.  
      \vspace{.25cm}
      
    \item  $(\id_1 \otimes/ \otimes \id_1 ) \circ(/ \otimes / ) \circ( \id_1 \otimes/ \otimes \id_1)$, see Figure~\ref{fig:ex1}.  This example can easily generalize to constructing morphisms depicted by any number of adjacent parallel straight arcs of any slope. See Figure~\ref{fig:ex1.2}. Such elements appear in relations $[T_{4}]-[T_{5}]'$ below.
\end{enumerate}

 \begin{figure}[H]
    \centering
\includegraphics[width=0.7\linewidth]{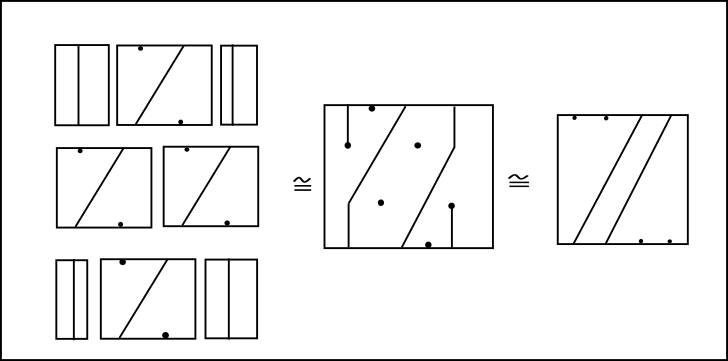}
    \caption{Creation of consecutive diagonals.}
    \label{fig:ex1}
\end{figure}
\begin{figure}[H]
    \centering
\includegraphics[width=0.6\linewidth]{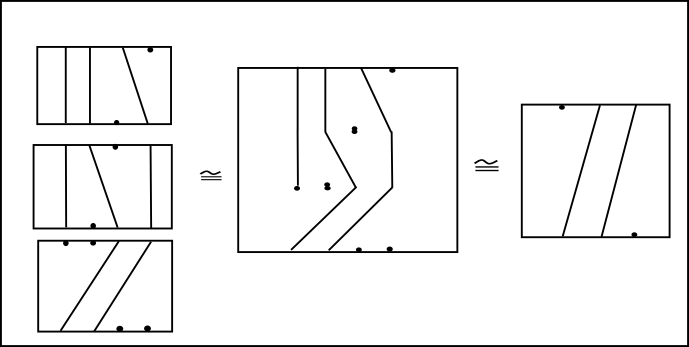}
    \caption{Creation of consecutive diagonals.}
    \label{fig:ex1.2}
\end{figure}

\end{examples}

 \subsection{Defining the extended unoriented  tangloid category $\Ti$ }
 
We are now ready to define the extended unoriented  tangloid category.

\begin{definition}[Extended unoriented  tangloid category]\label{def:eutc}
  Let $\Omega(\beta)$ be the free $\tfrac{1}{2}$--monoidal completion of the free path category $P(\beta^*)$ as defined in $UTC$. The \emph{extended unoriented  tangloid category},  denoted by $\Ti$, is the strict monoidal category formally presented by
\[
\FF\Big(\Omega(\beta)\big/\, \overline{\underline T}\Big),
\]
where $\overline{\underline T}$ is the $\tfrac{1}{2}$--monoidal congruence generated by  the {\it local  relations}  $[\mathrm{T}_i],[\mathrm{T}_i]'$ listed below with all their variants and derivations, along with the composition rules for $\emptyset$ in Definition~\ref{def:emptyset rules}. 

 \bigbreak

\begin{itemize}
    \item $[T_1]: (\cap)\circ X_- \sim \cap \sim (\cap )\circ X_+ $
\end{itemize}
     \vspace{.25cm}
     \begin{itemize}
         \item $[T_1]':(X_-)\circ \cup \sim \cup \sim (X_+) \circ \cup$
     \end{itemize}
     \vspace{.25cm}
    \begin{itemize}
    \item $[T_2]:(\cap\otimes \mathrm{id}_1)\circ(\mathrm{id}_1\otimes \cup) \sim (\emptyset \otimes /) \circ( / \otimes \emptyset)$
      \end{itemize}
   \vspace{.25cm}
   
    \begin{itemize}
  \item   $[T_2]': (\backslash \otimes\emptyset )\circ( \emptyset \otimes \backslash )\sim (\mathrm{id}_1 \otimes \cap)\circ(\cup\otimes \mathrm{id}_1)$
     \end{itemize}
     \vspace{.25cm}
     
    \begin{itemize}
    \item $[T_3]: (\backslash)\circ(/)\sim \id_1 \otimes \emptyset $
      \end{itemize}
   \vspace{.25cm}
   
    \begin{itemize}
    \item $[T_3]': \emptyset \otimes \id_1 \sim(/)\circ(\backslash)$
    \end{itemize}
   \vspace{.25cm}

\begin{itemize}
   \item $[T_{4}]: (\id \otimes\cap)\circ\Big( (\id_2\otimes \backslash )\circ(\id_1\otimes\backslash \otimes \id_1)\circ \big  ( ( \id_1\otimes/ \otimes \id_1)\circ(  /\otimes /  )\circ(\id_1 \otimes/\otimes \id_1) \big)\Big)\sim \cap \otimes  \emptyset$
      \end{itemize}
   \vspace{.25cm}
   
      \begin{itemize}
          \item $[T_{4}]': \emptyset_1\otimes \cap  \sim \,
  (\cap\otimes \id_1) \circ\Big ( (/ \otimes \id_2 )\circ(\id_1 \otimes/ \otimes \id_1 ) \circ \big ( ( \id_1\otimes \backslash \otimes \id_1)\circ(  \backslash\otimes \backslash  )\circ (\id_1 \otimes \backslash\otimes \id_1) \big) \Big)$
      \end{itemize}

   \vspace{.25cm}
\begin{itemize}
   \item $[T_{5}]: \Big( (\id_2\otimes \backslash )\circ(\id_1\otimes\backslash \otimes \id_1)\circ \big ( ( \id_1\otimes/ \otimes \id_1)\circ(  /\otimes /  )\circ(\id_1 \otimes/\otimes \id_1) \big)\Big) \circ(\cup  \otimes \emptyset_1) \sim \emptyset \otimes  \cup$
      \end{itemize}
         \vspace{.25cm}
      \begin{itemize}
          \item $ [T_{5}]': \cup\otimes \emptyset \sim \,
  \Big ( (/ \otimes \id_2 )\circ(\id_1 \otimes/ \otimes \id_1 )\circ \big ( ( \id_1\otimes \backslash \otimes \id_1)\circ(  \backslash\otimes \backslash  )\circ(\id_1 \otimes \backslash\otimes \id_1) \big) \Big)\circ(\emptyset \otimes \cup)  $
      \end{itemize}
   \vspace{.25cm}

    \begin{itemize}
     \item {$[T_6]:X_-\circ X_+\sim \mathrm{id}_2\sim X_+ \circ X_-$}
    \end{itemize}
    
   \vspace{.25cm}
   
    \begin{itemize}
    \item {$[T_7]: (X_-\otimes \mathrm{id}_1)\circ(\mathrm{id}_1\otimes X_-)\circ (X_-\otimes \mathrm{id}_1)\sim (\mathrm{id}_1 \otimes X_-)\circ (X_-\otimes \mathrm{id}_1)\circ (\mathrm{id}_1 \otimes X_-)$}
    \end{itemize}
   \vspace{.25cm}
\begin{itemize}
    \item  {$[T_{8}]:( \emptyset \otimes \backslash)\circ (\cap \otimes \mathrm{id}_1)\circ (\mathrm{id}_1\otimes X_+) \sim  ( /\otimes \emptyset)\circ (\mathrm{id}_1\otimes \cap)\circ (X_-\otimes \mathrm{id}_1)$}
      \end{itemize}
     \vspace{.25cm}

     \begin{itemize}
       \item $[T_{8}]': ( \emptyset \otimes \backslash )\circ(\cap\otimes \id_1)\circ(\id_1\otimes X_-)\sim ( / \otimes \emptyset ) \circ (\id_1\otimes \cap)\circ (X_+\otimes \id_1)$
    \end{itemize}
 \vspace{.25cm}

   \begin{itemize}
       \item $[T_{9}] : (\id_1 \otimes X_- )\circ (\cup \otimes \id_1 ) \circ ( \id_1 \otimes / ) \sim ( X_+ \otimes \id_1) \circ (\id_1 \otimes \cup) \circ ( \backslash \otimes \id_1)$
   \end{itemize}
    
   \vspace{.25cm}

\begin{itemize}
       \item $[T_{9}]' : (\id_1 \otimes X_+ )\circ (\cup \otimes \id_1 ) \circ ( \id_1 \otimes / ) \sim ( X_- \otimes \id_1) \circ (\id_1 \otimes \cup) \circ ( \backslash \otimes \id_1)$
   \end{itemize}
    
   \vspace{.25cm}

   \begin{itemize}
   \item $[T_{10}]: (\backslash \otimes \id_1)\circ ( \id_1 \otimes X_+) \circ ( / \otimes\id_1) \sim ( \id_1 \otimes / ) \circ ( X_+ \otimes \id_1 )\circ ( \id_1 \otimes \backslash)$

   \end{itemize}
 \vspace{.25cm}
   \begin{itemize}
   \item $[T_{11}]: (\backslash \otimes \id_1)\circ ( \id_1 \otimes X_-) \circ ( / \otimes\id_1) \sim ( \id_1 \otimes / ) \circ ( X_- \otimes \id_1 )\circ ( \id_1 \otimes \backslash)$

   \end{itemize}
%\begin{itemize}
 %   \item $[T_{10}]: (\backslash)\circ ( X_+) \sim \id_1 \otimes \emptyset \sim (\backslash)\circ( X_-)$
%\end{itemize}

 %  \vspace{.25cm}
%\begin{itemize}
%    \item $[T_{10}]': (/)\circ ( X_-) \sim \emptyset \otimes \id_1  \sim (/)\circ ( X_+)$
%\end{itemize}

   %\vspace{.25cm}
%\begin{itemize}
 %   \item $[T_{11}]:( X_+)\circ (\backslash) \sim  \emptyset \otimes\id_1  \sim ( X_-)\circ (\backslash)$
%\end{itemize}

%   \vspace{.25cm}
%\begin{itemize}
%    \item $[T_{11}]': ( X_+)\circ (/) \sim  \id_1  \otimes \emptyset \sim ( X_-) \circ (/)$
%\end{itemize}
\vspace{.25cm}
\begin{itemize}
 \item  $[T_{12}]:  \cap\, \otimes\, \cup\sim \, \cup\circ\, \cap \sim \,\cup \otimes\, \cap$
    \end{itemize}
     \vspace{.25cm}
\begin{itemize}
 \item  $[T_{13}]:  \text{¡}\, \otimes\, !\sim \, !\circ\, \text{¡} \sim \,! \otimes\, \text{¡}$
    \end{itemize}

 \vspace{.25cm}

    \begin{itemize}
 \item  $[T_{14}]:  \cap\, \otimes\, !\sim \, (\emptyset \,\otimes\, !)\circ\, \cap \sim \, (! \,\otimes\, \emptyset)\circ\, \cap  \sim \,! \otimes\, \cap$
    \end{itemize}
    \vspace{.25cm}
     \begin{itemize}
 \item  $[T_{14}]':  \text{¡}\, \otimes\, \cup\sim \, \cup \circ (\emptyset \,\otimes\, \text{¡}) \sim \, \cup \circ (\text{¡}\,\otimes\, \emptyset)\circ\, \cup  \sim \, \cup  \otimes\, \text{¡}$
    \end{itemize}

   \vspace{.25cm}
    \begin{itemize}
    \item {$[T_{15}]:\cap \circ (\mathrm{id}_1\otimes !)\sim\, \text{¡} \otimes \emptyset  $}
    \end{itemize}

    \vspace{.25cm}
 \begin{itemize}
     \item $[T_{15}]': \emptyset \otimes\,\text{¡}  \sim\,\cap \circ (!\otimes \mathrm{id}_1)$
 \end{itemize}

    \vspace{.25cm}
    \begin{itemize}
    \item {$ [T_{16}]: (\mathrm{id}_1\otimes \text{¡}) \circ\cup\sim\,\,  ! \otimes \emptyset$}
\end{itemize}

        \vspace{.25cm}
          \begin{itemize}
     \item  $[T_{16}]':  \emptyset\otimes\, !\sim (\text{¡} \otimes \mathrm{id}_1)\circ\cup $
\end{itemize}

        \vspace{.25cm}
          \begin{itemize}
       \item {$[T_{17}]: (\emptyset \otimes\text{¡})\circ(/)\sim\, \text{¡} \otimes \emptyset  $}
\end{itemize}

          \vspace{.25cm}
            \begin{itemize}
       \item $[T_{17}]': \emptyset \otimes \, \text{¡} \sim\,(\text{¡} \otimes \emptyset)\circ (\backslash )$
\end{itemize}
 \vspace{.25cm}
 
         \begin{itemize}
 \item {$ [T_{18}]: (\backslash)\circ(\emptyset\,\otimes\, !) \sim\,\,  !\,\otimes\, \emptyset   $}
   \end{itemize}
   \vspace{.25cm}

   \begin{itemize}
 \item  $[T_{18}]':  \emptyset\, \otimes\, !\sim \, (/)\circ(! \otimes \emptyset)$
    \end{itemize}
    \vspace{.25cm}

\end{definition}
% \bigbreak

\noindent The above local relations can be presented diagrammatically  as exemplified in Figure~\ref{fig:placeholder}  (recall that we read the diagram from top to bottom), where the composition  and tensoring rules also apply.

\begin{figure}[H]
    \centering
    \includegraphics[width=0.8\linewidth]{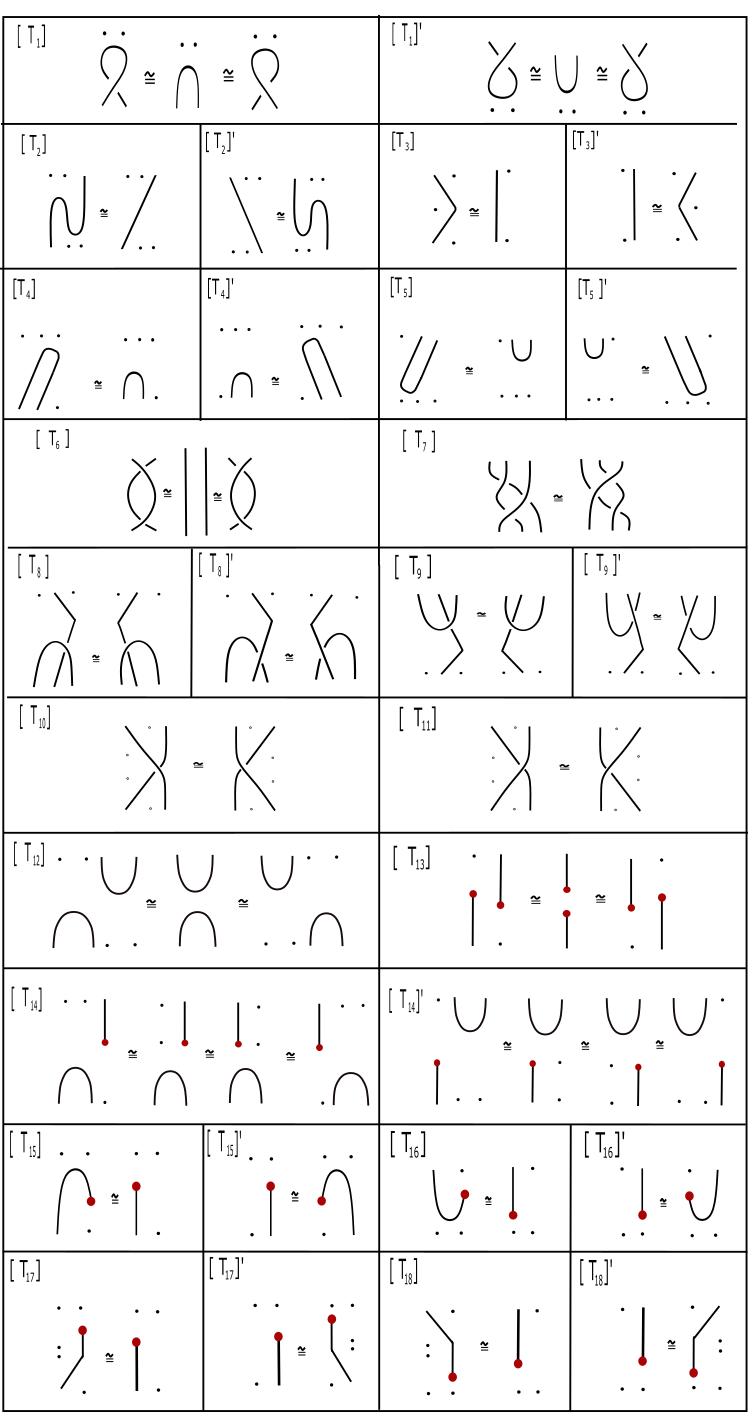}
    \caption{Local  relations in $\Ti$.}
    \label{fig:placeholder}
\end{figure}

As in $UTC$, it turns out that having the swing moves $[T_8]-[T_9]^\prime$ in our list of isotopy moves makes 
redundant certain instances of Reidemeister moves involving horizontal arcs. But in $\Ti$ we may need to compose with $/,\backslash$ in order to be able to apply the swing moves, as for example in the case of a horizontal RI move that can be obtained from a vertical one and two swing moves. 

Also, as in $UTC$,  since $\Ti$ is strict monoidal, we freely use the \emph{interchange law}:
for any composable morphisms $f,f',g,g'$ we have
$(f\otimes g)\circ(f'\otimes g')=(f\circ f')\otimes(g\circ g')$.
 In particular,  morphisms supported on disjoint tensor factors commute. 
  For example, assuming composable identity morphisms we have:
\begin{align}  \label{general_far_commuting2}
    (f\otimes \id) \circ (\id\otimes g)  = (f \circ \id) \otimes (\id \circ  g)  = f\otimes g= (\id \circ f) \otimes (g \circ \id ) = (\id \otimes g) \circ (f \otimes \id)
\end{align} 
\begin{figure}[H]
    \centering
    \includegraphics[width=0.35\linewidth]{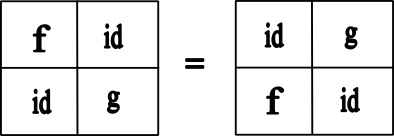}
    %\caption{Caption}
    \label{fig:}
\end{figure}
We do not list these ``far-commuting'' relations separately. 

\begin{remark}
Thanks to the interchange law and the identity morphism, we obtain the exchange of horizontal levels in the diagrammatic context. For instance, we obtain the exchange of horizontal levels of  $\cup$,  $\cap$ or a crossing with an endpoint morphism (see Figure~\ref{fig:far-commute-re}): 
\begin{align}  \label{far_commuting_endpoint_ex}
 (\cap \otimes  \id)  \circ  (\id_2 \, \otimes \, !) =  (\cap \circ \id_2) \otimes   (\id \, \circ\, !)= (\cap \circ \id_2) \otimes   (\,!\,)= \cap\, \otimes\, ! 
\end{align}
\begin{align}  \label{far_commuting_endpoint_ex2}
(\id_2 \, \otimes \text{¡} \, ) \circ  (\cup \otimes  \id)     =  (\id_2 \circ \cup) \otimes   (\text{¡}\, \circ\, \id)=(\id_2 \circ \cup) \otimes   (\text{¡}\,)  = \cup\, \otimes\, \text{¡}
\end{align}
\begin{align}  \label{}
 (X_{\pm} \otimes  \id)  \circ  (\id_2 \, \otimes \, !) =  (X_{\pm} \circ \id_2) \otimes   (\id \, \circ\, !) = (X_{\pm} \circ \id_2) \otimes   (\, !\,) =X_{\pm}\, \otimes\, ! 
\end{align}

%\begin{align}
 %   !\, (\emptyset_2\,\otimes\, !) \circ (\emptyset\otimes \cap) \circ (\, \text{¡}\, \otimes \id_2 )\circ (\cup \otimes \id )&= \, (\emptyset_2 \, \otimes\,!)\,\circ (\emptyset\, \otimes \cap) \circ (\, \text{¡}\, \otimes (\id \otimes \id)  )\circ (\cup \otimes \id ) \\
%    &= \, (\emptyset_2 \, \otimes\,!)\,\circ (\emptyset\, \otimes \cap) \circ ((\, \text{¡}\, \otimes \id )\otimes \id  )\circ (\cup \otimes \id ) \\
  %  &= \, (\emptyset_2 \, \otimes\,!)\,\circ (\emptyset\, \otimes \cap) \circ  ((\, \text{¡}\, \otimes \id )\circ \cup)\otimes (\id \circ \id)\\
  %  &= \, (\emptyset_2 \, \otimes\,!)\,\circ (\emptyset\, \otimes \cap) \,\circ ( (\emptyset\, \otimes \,! )\,\otimes  \id)) \\
  %  &= \, (\emptyset_2 \, \otimes\,!)\,\circ (\emptyset\, \otimes \cap) \,\circ  (\emptyset\, \otimes (\,! \,\otimes  \id))\\
 %   &= \, (\emptyset_2 \, \otimes\,!)\,\circ ((\emptyset\, \circ \emptyset)\, \otimes (\cap\, \circ (!\otimes \id)))\\
  %  &= \, (\emptyset_2 \, \otimes\,!)\,\circ (\emptyset\,\otimes (\emptyset\,\otimes \text{¡}))\\
  %  &= \, (\emptyset_2 \, \otimes\,!)\,\circ (\emptyset_2\,\otimes \,\text{¡})  = \emptyset_2 \, \otimes\,( !\, \circ \,  \text{¡} \,)
%\end{align}

%\begin{align}
 %   (! \otimes\, \emptyset_2)\, \circ ( \cap\,\otimes\emptyset) \circ ( \id_2 \otimes \text{¡} )\circ (\id \otimes \cup) =\, (!\, \circ \,  \text{¡}) \otimes\, \emptyset_2
%\end{align}

\noindent Note  that Equation (\ref{endpoint_shift}) in $UTC$ about horizontal shifts of endpoints also holds in $\Ti$ by virtue of the more elementary relations $[T_{17}]$--$[T_{18}]'$ in $\Ti$.
\end{remark}

\begin{figure}[H]
    \centering
\includegraphics[width=0.42\linewidth]{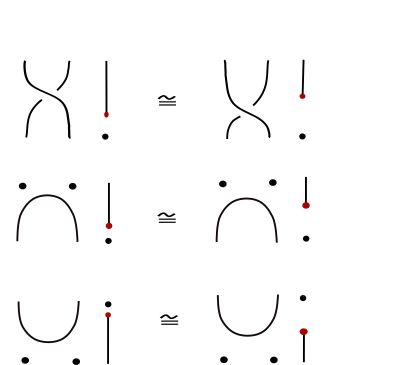}
    \caption{Geometric representations for the vertical shifts of endpoints. }
    \label{fig:far-commute-re}
\end{figure}

\begin{remark}\label{re:shifting of diagonal}
    The vertical shifting of diagonals, showing in Figure~\ref{fig:shifting_diagonal}, is obvious because: 
    \[ \backslash \circ \id_2= \id_2 \circ \backslash\]
     \[ / \circ \id_2= \id_2 \circ /\]
\end{remark}
\begin{figure}[H]
    \centering
\includegraphics[width=0.4\linewidth]{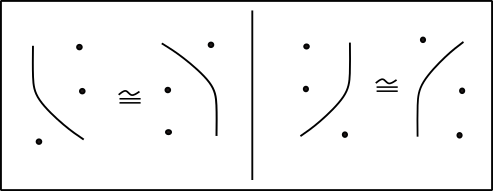}
    \caption{Vertical shifting of diagonals.}
    \label{fig:shifting_diagonal}
\end{figure}

\begin{lemma} \label{crossing_hshift}
The symmetric planar isotopy relations,  which enable horizontal shifts  of crossings with four or more ends can be derived from the basic   $[T_{10}], [T_{11}]$ $\Ti$  relations and the vertical shifting of diagonals. The relations with four ends is depicted in Figure~\ref{fig:PIR_of_crossing} and reads as follows:
\begin{align*}
    (\backslash\otimes \id_2 ) \circ (\id_1 \otimes \backslash \otimes \id_1 ) \circ (\id_2 \otimes X_{\pm})&\circ ( \id_1 \otimes / \otimes \id_1 ) \circ (/ \otimes \id_2 ) \\&=  (\backslash \otimes / ) \circ (\id_1 \otimes X_\pm \otimes \id_1) \circ (/ \otimes \backslash)\\&=  ( \id_2 \otimes / ) \circ  ( \id_1 \otimes / \otimes \id_1) \circ ( X_\pm \otimes \id_2 )\circ ( \id_1 \otimes \backslash \otimes \id_1) \circ (\id_2 \otimes \backslash)
\end{align*}
\begin{figure}[H]
    \centering
    \includegraphics[width=0.45\linewidth]{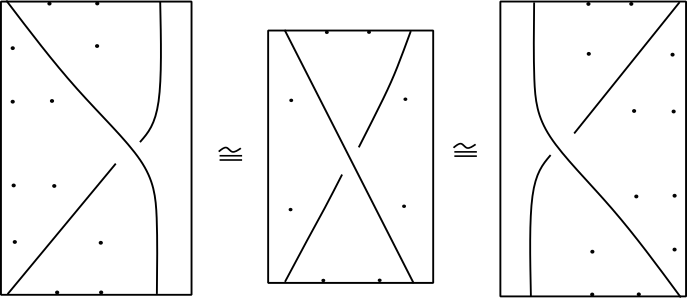}
    \caption{The horizontal shiftings of a crossing.}
    \label{fig:PIR_of_crossing}
\end{figure}
\end{lemma}

\begin{proof} We will prove the validity of the relations with four ends. The general case follows by an obvious iteration.
\begin{align*}
    (\backslash \otimes / ) \circ (\id_1 \otimes X_\pm \otimes \id_1) \circ (/ \otimes \backslash) &= (\backslash \otimes / ) \circ (\id_2 \otimes \id_2) \circ  (\id_1 \otimes X_\pm \otimes \id_1) \circ (\id_2 \otimes \id_2) \circ  (/ \otimes \backslash)\\
     &\stackrel{interchange\, law}{=}  (\backslash \circ  \id_2 ) \otimes (/  \circ  \id_2) \circ  (\id_1 \otimes X_\pm \otimes \id_1) \circ (\id_2 \circ / ) \otimes  (\id_2  \circ  \backslash)\\&\stackrel{Remark~\ref{re:shifting of diagonal}}{=} (\backslash \circ  \id_2 ) \otimes (\id_2  \circ  /) \circ  (\id_1 \otimes X_\pm \otimes \id_1) \circ (\id_2 \circ / ) \otimes  (\backslash \circ  \id_2)\\
     &\stackrel{interchange\, law}{=} (\backslash \otimes  \id_2 ) \circ (\id_2  \otimes  /) \circ  (\id_1 \otimes X_\pm \otimes \id_1) \circ (\id_2 \otimes \backslash ) \circ  (/  \otimes  \id_2)\\
     &\stackrel{[T_{10}]}{=} (\backslash \otimes  \id_2 ) \circ (\id_1  \otimes  \backslash \otimes \id_1) \circ  (\id_2 \otimes X_\pm ) \circ (\id_1 \otimes / \otimes \id_1 ) \circ  (/  \otimes  \id_2)\\
    \end{align*}

    \begin{align*}
    (\backslash \otimes / ) \circ (\id_1 \otimes X_\pm \otimes \id_1) \circ (/ \otimes \backslash) &= (\backslash \otimes / ) \circ (\id_2 \otimes \id_2) \circ  (\id_1 \otimes X_\pm \otimes \id_1) \circ (\id_2 \otimes \id_2) \circ  (/ \otimes \backslash)\\
     &\stackrel{interchange\, law}{=}  (\backslash \circ  \id_2 ) \otimes (/  \circ  \id_2) \circ  (\id_1 \otimes X_\pm \otimes \id_1) \circ (\id_2 \circ / ) \otimes  (\id_2  \circ  \backslash)\\&\stackrel{Remark~\ref{re:shifting of diagonal}}{=} (\id_2 \circ  \backslash ) \otimes (/  \circ  \id_2) \circ  (\id_1 \otimes X_\pm \otimes \id_1) \circ (/ \circ \id_2 ) \otimes  (\id_2 \circ  \backslash)\\
     &\stackrel{interchange\, law}{=} (\id_2 \otimes  / ) \circ (\backslash  \otimes  \id_2) \circ  (\id_1 \otimes X_\pm \otimes \id_1) \circ (/ \otimes \id_2 ) \circ  (\id_2 \otimes  \backslash)\\
     &\stackrel{[T_{10}]}{=} (\id_2 \otimes  / ) \circ (\id_1  \otimes  / \otimes \id_1) \circ  ( X_\pm \otimes \id_2 ) \circ (\id_1 \otimes \backslash \otimes \id_1 ) \circ  (\id_2 \otimes \backslash  )\\
    \end{align*}
\end{proof}

We conclude this section with the following definition in analogy to  Definition~\ref{def:regular-framed_UTC} for $UTC$. 

\begin{definition}\label{def:regular-framed_Tangi}
 We define the  \emph{regular extended unoriented tangloid category}, denoted $\rTi$, to have the same generators and relations as $\Ti$, except for the relations $[T_1]$ and $[T_1]'$ which are omitted. 
 Furthermore, $\Ti$ gives rise to the \emph{framed extended unoriented tangloid category}, denoted $\fTi$, which has the same generators and relations as $\Ti$, except for the relations $[T_1]$ and $[T_1]'$ which are replaced by the framing preserving relations depicted in Figure~\ref{fig:framed-ex R1}. 
\end{definition}

\begin{figure}[H]
    \centering
\includegraphics[width=0.6\linewidth]{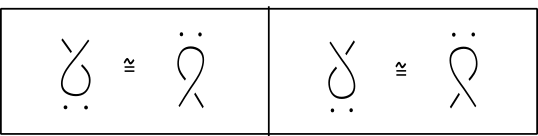}
    \caption{The framing preserving relations in $\fTi$. }
    \label{fig:framed-ex R1}
\end{figure}

%%%%%%%%%%%%%%%%%%%%%%%%%%%%%%%%%%%%%%%%%%%%%%%%%%%%%%%%%%%%%%%%%%%%

\section{The Subcategory of Braidoids}\label{sec:subcategory_braidoids}

In this subsection, we explain how the notion of  braidoid, introduced by
Gügümcü and Lambropoulou \cite{gugumcu2021braidoids}, can be interpreted inside the unoriented tangloid category $UTC$ and the extended unoriented
tangloid category $\Ti$ as a special class of morphisms  forming
a natural diagrammatic subcategory.

A \emph{braidoid diagram} is a braid-like diagram drawn in the square
$I \times I \subset \mathbb{R}^2$, where $I=[0,1]$ and the vertical direction is in \cite{gugumcu2021braidoids} oriented downward. It consists of a finite number of arcs, called strands, with
only finitely many transversal double points, each equipped with over/under
crossing data. As in classical braid diagrams, every strand is oriented
downward and has no local maxima or minima; equivalently, each horizontal line
intersects a strand at most once. The main difference from classical braid
diagrams is that a braidoid diagram may contain free strands ending at two
distinguished endpoints, called the \emph{leg} and the \emph{head}. The leg is
the initial endpoint of a free strand, while the head is the terminal endpoint.
The endpoints lying on the top and bottom boundary lines are called
\emph{braidoid ends}. These braidoid ends occur in corresponding pairs, one on
the top boundary and one on the bottom boundary. Thus, a braidoid diagram
generalizes a classical braid diagram by allowing two free strands with two special internal
endpoints. 

The isotopy classes of such diagrams under the allowed braidoid
moves, are called \emph{braidoids}.
 The allowed moves for braidoid diagrams are the downward-oriented Reidemeister
moves $\Omega_2$ and $\Omega_3$, planar $\Delta$-moves respecting the monotonicity, and special moves of the
endpoints called vertical moves and swing moves, which allow them to move within their diagrammatic region. However, as in knotoid theory,
there are also forbidden moves: the leg or the head is not allowed to pass over
or under another strand. These forbidden moves are essential, because allowing
them would trivialize the braiding information carried by the free strands.

\smallbreak
Braidoids can be realized inside  $UTC$ and $\Ti$ by restricting to those
morphisms which are braid-like. More precisely, 

\begin{definition} \label{def:braidoid category}
    The {\it braidoid category} is defined as the subcategory $\mathbf{Brd}$ of  $UTC$  with generating set $\{X_+,X_-, !, \text{¡}\}$  and relations the set of relations of $UTC$ from which are excluded all relations involving the generators $\cup, \cap$.   
\end{definition}

\begin{definition} \label{def:extended braidoid category}
The {\it extended braidoid category} is defined as the subcategory
\[
\mathbf{Brd}_{\Ti} \subset\Ti
\]
 with generating set $\{X_+,X_-,!, \text{¡}, \emptyset\}$  and relations the set of relations of $\Ti$ from which are excluded all relations involving the generators $\cup, \cap$. 
\end{definition}
By Definitions~\ref{def:braidoid category} and~\ref{def:extended braidoid category}, we have the inclusion functors 
\[
\iota:\mathbf{Brd}\hookrightarrow UTC \quad  \mbox{and}  \quad \iota:\mathbf{Brd}_{\Ti}\hookrightarrow \Ti.
\]
The objects in $\mathbf{Brd}$ and $\mathbf{Brd}_{\Ti}$ are natural numbers, interpreted geometrically as the number of
boundary positions. The morphisms are those morphisms of $UTC$ resp. $\Ti$ which
can be represented by diagrams satisfying the definition of braidoid.
%\begin{enumerate}
 %   \item all strands are monotone with respect to the downward/upward direction;
  %  \item the diagram has no local maxima or minima;
   % \item the diagram contains exactly one leg and one head;
   % \item the  boundary endpoints occur in corresponding top-bottom
   % pairs;
  %  \item the endpoint-forbidden moves are not imposed as relations.
% \end{enumerate} 
Under this interpretation, the elementary pieces of a braidoid diagram are sent to the corresponding  generators of $\mathbf{Brd}$ resp. $\mathbf{Brd}_{\Ti}$  in the following way: the positive and negative braid crossings are represented by the crossing generators $X_+$ and $X_-$:
\[
\sigma_i^+ \longmapsto
\id_{i-1} \otimes X_+ \otimes
\id_{n-i-1}
\]
\[
\sigma_i^- \longmapsto
\id_{i-1} \otimes X_- \otimes
\id_{n-i-1}
\]
 The head and leg of a braidoid are
represented by the endpoint generators
$\text{\text{¡}}$ and $!$ respectively:
\[
h_i\longmapsto  \id_{i-1}\otimes\, ! \otimes
\id_{n-i}
\]
\[
l_i\longmapsto  \id_{i-1}\otimes\, \text{¡} \otimes
\id_{n-i}
\]
The empty generator $\emptyset$ of $\Ti$ 
plays the role of an implicit point or placeholder in the diagram  allowing for the definition of a set of generators in the set of braidoids on $n$ strands under  multiplication operation the concatenation (vertical stacking) of diagrams. 
It follows that the morphisms (including the generators) in $\mathbf{Brd}_{\Ti}$ are in $\Hom(n,n)$, $n\in \N$. 

The diagonal elements $/$ and $\backslash$ shall be encoded as: 
\[
d_i\longmapsto  \id_{i-1}\otimes\, / \otimes
\id_{n-i-1}
\]
\[
\bar{d_i}\longmapsto  \id_{i-1}\otimes\, \backslash \otimes
\id_{n-i-1}
\]

Vertical stacking of braidoid diagrams corresponds to categorical
composition in $\Ti$, while horizontal juxtaposition corresponds to
the tensor product. Therefore, the assignment from braidoid diagrams to $UTC$-morphisms or 
$\Ti$-morphisms preserves both composition and tensor product. 
Furthermore, the relations in $\mathbf{Brd}$ and  $\mathbf{Brd}_{\Ti}$  correspond precisely to the isotopy relations in the set of braidoid diagrams. 
 Hence, the categories $\mathbf{Brd}$ and  $\mathbf{Brd}_{\Ti}$ of braidoids may be regarded as the diagrammatic
subcategories of $UTC$ and  $\Ti$ respectively, which furnish the natural categorical frameworks for the class of braidoids, and whose morphisms are monotone
tangloid diagrams. In this sense, the extended
tangloid category provides a larger ambient category, while braidoids are
obtained by imposing the additional global condition that the diagrams are
braid-like and by keeping the endpoint-forbidden moves excluded.

%%%%%%%%%%%%%%%%%%%%%%%%%%%%%%
\begin{remark}
The term `implicit point' was introduced in \cite{gugumcu2017knotoids,Gugumcu2017} as an auxiliary
combinatorial structure in an attempt  to express braidoids in terms of basic generating blocks. In \cite{Gugumcu2017} a systematic effort is made for listing basic relations among generating blocks of braidoids with operation the vertical stacking.  
Otherwise,  the empty generators and the diagonal generators $/$ and $\backslash$ were already introduced in the definition of the Motzkin algebra in \cite{Benkart2011Motzkin}.
\end{remark}
%%%%%%%%%%%%%%%%%%%%%%%%%%%%%%5
\begin{remark}
The categories $UTC$ and $\Ti$ contain, as natural subcategories, the categories
of knotoids, multi-knotoids, which are immersions of the unit interval and a number of circles in an oriented surface, and linkoids, immersions of some copies of the unit interval  in an oriented surface.  Thus, the tangloid framework
provides new categorical models for these diagrammatic theories.
\end{remark}

%%%%%%%%%%%%%%%%%%%%%%%%%%%%%%%%%%%%%%%%%%%%%%%%%%%%%%%%%%%%%%%%%%%%%%%%%

\section{The Tangloid Algebra} \label{sec:tangloid-algebra}

The aim of this section is to define a tangloid algebra related to the extended unoriented tangloid category $\Ti$ of Definition~\ref{def:eutc}. Our construction is analogous in spirit to other diagrammatic settings where categorical morphisms give rise to natural endomorphism algebras; compare, for instance, the Brauer category framework in~\cite{lehrer2015brauer}. 

%We denote by $R \, \Ti$  the $R$–linear tangloid category, for  a ring $R$.

\begin{definition}[Tangloid algebra $\mathcal{T}_n$]\label{def:tangloid-algebra}
Let $\Ti$ be the  extended tangloid category of Definition~\ref{def:eutc}.  For $n\ge 0$ the \emph{tangloid algebra of degree $n$} is defined as the
$\mathbb{C}$–algebra
\[
  \mathcal{T}_n \, := \,\C \, \mathrm{End}_{{\Ti}}(n)\,\
\]
where \[\mathrm{End}_{{\Ti}}(n) := \Hom_\Ti(n,n) \]

\noindent The factor $\mathbb C$ indicates the linearization of the endomorphism set $\End_{\Ti}(n)$. 
\end{definition} 

  Equivalently we may state: 

\begin{proposition}\label{prop:TA-diagram-realization}
For each $n\ge 0$, the algebra $\mathcal{T}_n=\C\,\End_{\Ti}(n)$ is naturally identified with the unital diagram algebra over $\mathbb{C}$  whose
\textit{linear basis} as a $\mathbb C$-vector space consists of all equivalence classes of tangloid diagrams in a rectangle with $n$ ordered boundary points on the top and $n$ on the bottom, some of which may merely act as placeholders, modulo the local relations of $\Ti$ in  Definition~\ref{def:eutc}.  Under this identification, the algebra multiplication uses the same rule as the composition rule for $\Ti$, induced by vertical stacking  of diagrams (from top to bottom), addition is formal summation, 
and the unit, $\id_n$, is the identity diagram on $n$ vertical strands.
\end{proposition}

\begin{proof}
By definition, $\End_{\Ti}(n)=\Hom_{\Ti}(n,n)$ is the set of morphisms from $n$ to $n$ in the extended  tangloid  category $\Ti$. 
Recall that, in the diagrammatic calculus underlying  $\Ti$, cups, caps, endpoint maps and
diagonal connectors  create ``empty'' boundary positions (no arc emanating from that position). Since $\Ti$ is presented diagrammatically, its morphisms are represented by tangloid diagrams modulo the defining local relations, which include the composition rules for one or more empty positions.  By definition $\mathcal{T}_n$ adopts all local relations of $\Ti$. 
Hence $\Hom_{\Ti}(n,n)$ is precisely the set of $n\to n$ tangloid diagrams modulo these relations and is closed under composition. 
The multiplication in $\mathcal{T}_n$ is the composition in $\Ti$, which is realized
diagrammatically by vertical stacking  (from top to bottom). The unit is the identity morphism,
$\id_n$, represented by $n$ vertical strands. Passing to the $\mathbb C$-linear span $\mathbb C\,\End_{\Ti}(n)$ allows horizontal composition in the form of finite complex linear combinations
of diagrams, giving the claimed vector space. 
\end{proof}

The diagrams in $\mathcal{T}_n$ are built from crossings, cups, caps, endpoint segments and empty positions. In $\mathcal{T}_n$ we use the notations $x, \bar x, u,a,h,l,, \phi$ instead of $X_+,X_-,\cup,\cap, , \text{¡}, \bangup, \emptyset$ respectively, undergoing vertical composition.  More precisely, in $\mathcal{T}_n$ we distinguish the following  elements, depicted in Figure \ref{fig:g_TA}:
\smallbreak
\begin{itemize}
  \item for $1\le i\le n-1$, let $x_i$ (resp.\  $\bar x_i$) denote  the
        positive (resp.  negative) crossings,  between the $i$-th and $(i{+}1)$-st
        strands;
        \bigbreak
  \item for $1\le i\le n-1$, let $u_i$ (resp.\ $a_i$) denote  the elementary 
  cup (resp. cap) diagram connecting the top (resp. bottom) $i$-th and $(i{+}1)$-st boundary points (and all other strands vertical);
        \bigbreak

  \item for $1\le i\le n$, let $h_i$ (resp.\ $l_i$) denote the diagram
        obtained by inserting the endpoint map $\text{¡}$
        (resp.\ $!$ ) at the $i$-th position, with all other
        strands vertical.
        \bigbreak
        \item for $1\le i\le n$, let $\phi_i$ denote the diagram obtained by inserting the empty set map $(\emptyset)$  at the $i$-th position, with all other strands vertical.
\end{itemize} 

\smallbreak
\noindent The elements $x_i,\bar x_i, u_i,a_i, h_i, l_i, \phi_i$ generate $\mathcal{T}_n$ as
a $\mathbb{C}$–algebra. All generators except $x_i,\bar x_i$  create boundary empty positions. The need for including $\phi_i$ in the generators' set arises from the fact that the tensoring operation in $\Ti$ is not available in $\mathcal{T}_n$, as it does not preserve the degree $n$. 

\begin{figure}[H]
    \centering
    \includegraphics[width=.9\linewidth]{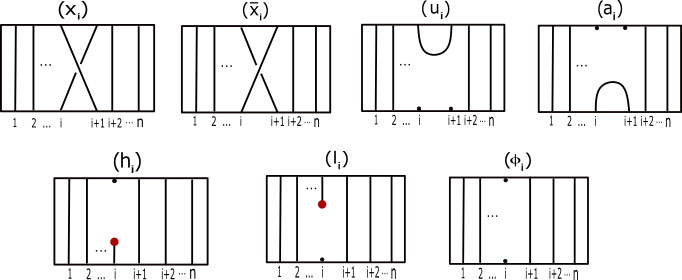}
    \caption{Geometric representatives of the tangloid algebra generators.}
    \label{fig:g_TA}
\end{figure}

\textbf{Note}, as in $\Ti$, we use the following morphisms \begin{itemize}
     \item for $1\le i\le n-1$, let $d_i$ denote the 
        diagonal connector diagram which joins $(i,\mathrm{bottom})$ to
        $(i{+}1,\mathrm{top})$  and all
        other strands vertical;
        \bigbreak
         \item for $1\le i\le n-1$, let $\bar d_i$ denote the 
        diagonal connector diagram which joins $(i{+}1,\mathrm{bottom})$ to
        $(i,\mathrm{top})$   and all
        other strands vertical.
          \bigbreak
\end{itemize}

Although these diagonal morphisms are not generators, they can be obtained by composing a crossing with the placeholder morphism. We introduce the notation $d_i$ and $\bar d_i$ for these morphisms and use it to simplify the defining relations. See illustration in Figure~\ref{fig:diagonal_morphisms_TA}. 
\begin{figure}[H]
    \centering
    \includegraphics[width=0.42\linewidth]{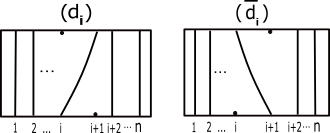}
    \caption{Diagonal morphisms in the tangloid algebra.}
    \label{fig:diagonal_morphisms_TA}
\end{figure}
Definition~\ref{def:emptyset rules} implies the following composition rules involving the generators $\phi_i$.

 \subsection{Composition rules involving the generators $\phi_i$} \label{ss:composition rules empty}
 
\begin{enumerate}
\item $\phi_i$ composed with $\phi_i$ is contracted to $\phi_i$: $\phi_i \, \circ \phi_i = \phi_i$.

    \item \label{item:1}$\id_n$ composed with $\phi_i$ is contracted to $\phi_i$:\begin{align*}
        \id_n \, \circ\phi_i = \phi_i = \phi_i \, \circ \id_n
    \end{align*} 
 
      \item  \label{item:2}
    \begin{align*}
        ( \phi_i  \circ   \phi_{i+1})\circ\,u_i  & =  \phi_i\,\circ u_i= \phi_{i+1}\,\circ u_i = \phi_i  \circ   \phi_{i+1}  \\
        \ u_i \circ\,(\phi_i  \circ   \phi_{i+1}) \,  & = u_i\,\circ\phi_i  =u_i\,\circ    \phi_{i+1}=u_i \\
        u_i\circ u_i&=u_i
    \end{align*}
    
      \item \label{item:3}
      \begin{align*}
          a_i \, \circ (\phi_i  \circ   \phi_{i+1}) & =a_i  \,  \circ \phi_i  =a_i  \, \circ (\phi_{i+1})= \phi_i  \circ   \phi_{i+1} \\
          (\phi_i  \circ   \phi_{i+1}) \, \circ a_i  & = \phi_i\, \circ a_i  =\phi_{i+1}\,\circ a_i = a_i \\
          a_i\circ a_i &=a_i
      \end{align*}
      
     \item \label{item:4} \begin{align*}
         x_i \, \circ (\phi_i  \circ   \phi_{i+1}) & = (\phi_i  \circ   \phi_{i+1})  = (\phi_i  \circ   \phi_{i+1}) \circ \, x_i\\
         \bar x_i \, \circ (\phi_i  \circ   \phi_{i+1}) & = (\phi_i  \circ   \phi_{i+1})  = (\phi_i  \circ   \phi_{i+1}) \circ \, \bar x_i\\
         \phi_{i+1}\circ x_i &= \phi_{i+1} \circ \bar x_i\\
         \phi_{i+1}\circ x_i &=x_i\circ \phi_i\\
         \phi_{i+1}\circ x_i &= \bar x_i\circ \phi_i\\
         \phi_{i}\circ \, x_i &= \phi_{i}\circ \, \bar x_i\\
         \phi_{i}\circ \, x_i &=x_i \circ\phi_{i+1}\\
          \phi_{i}\circ \, x_i &= \bar x_i \circ\phi_{i+1}
     \end{align*}

     \item \label{item:8} \begin{align*}
         h_i\circ\phi_i&= \phi_i\\
         \phi_i\circ h_i &=h_i
     \end{align*}
     
     \item \label{item:9} \begin{align*}\label{eq:9}
         l_i\circ \phi_i&=\,\, l_i\\
         \phi_i\circ  l_i &= \, \phi_i
     \end{align*}
     
     \item \label{item:10}\begin{align*}
   ( \phi_i \circ   h_{i+1} ) \circ x_i &=  ( \phi_i \circ   h_{i+1} )\circ  \bar x_i \\
   ( h_i \circ   \phi_{i+1} ) \circ x_i & =   ( h_i \circ   \phi_{i+1} ) \circ  \bar x_i\\
   x_i \circ  ( \phi_i \circ h_{i+1} ) &= ( \phi_i \circ h_{i+1} )=  \bar x_i \circ ( \phi_i \circ h_{i+1} ) \\
  x_i \circ ( h_i \circ   \phi_{i+1} )& =  h_i \circ   \phi_{i+1}  =  \bar x_i \circ ( h_i \circ   \phi_{i+1} )
     \end{align*}
     
    \item  \label{item:11}\begin{align*}
   ( \phi_i \, \circ   l_{i+1} ) \circ x_i &= ( \phi_i \, \circ   l_{i+1} )  = ( \phi_i \, \circ   l_{i+1} ) \circ  \bar x_i \\
   ( l_i \circ  \phi_{i+1} ) \circ x_i&=  ( l_i \circ  \phi_{i+1} ) \,  =  ( l_i \circ  \phi_{i+1} ) \circ \bar x_i\\
   x_i \circ ( \phi_i \, \circ   l_{i+1} )  &  =  \bar x_i \circ ( \phi_i \, \circ   l_{i+1} ) \\
  x_i \circ  ( l_i \circ  \phi_{i+1} )& =  \bar x_i \circ  ( l_i \circ  \phi_{i+1} )
     \end{align*}

     \item \label{item:12} \begin{align*}
        ( h_i \circ   \phi_{i+1} )\,\circ u_i  & = (\phi_i  \circ   \phi_{i+1})=( \phi_i
    \circ   h_{i+1} ) \circ u_i  \\
        u_i \,\circ   ( l_i \circ  \phi_{i+1} ) & = u_i = u_i\circ ( \phi_i \, \circ   l_{i+1} )\\
          ( l_i \circ  \phi_{i+1} ) \circ  u_i \, & =\,  l_i \circ  \phi_{i+1} \\
           ( \phi_i \, \circ   l_{i+1} ) \circ u_i &= \phi_i \, \circ   l_{i+1} 
    \end{align*}
    
    \item  \label{item:13}\begin{align*}
       a_i \,\circ ( h_i \circ   \phi_{i+1} ) \,  & =  ( h_i \circ   \phi_{i+1} ) \\
       a_i \,\circ  ( \phi_i \circ h_{i+1} ) &= ( \phi_i \circ h_{i+1} ) \\
        ( h_i \circ   \phi_{i+1} ) \,\circ a_i  & = a_i =( \phi_i \circ h_{i+1} )\, \circ a_i \\
        a_i \,\circ ( l_i \circ  \phi_{i+1} )\,  & = \phi_i  \circ   \phi_{i+1} =a_i \circ\, (\phi_i \, \circ   l_{i+1} ) 
    \end{align*}

\end{enumerate}

\subsection{The diagonal elements}

For the diagonal elements we have the composition rules below,  which follow from the  ones listed in Subsection~\ref{ss:composition rules empty}.

\begin{enumerate}
\item \label{item:6} \begin{align*}
        d_i\, \circ (\phi_i  \circ   \phi_{i+1})  & = (\phi_i  \circ   \phi_{i+1})  =  (\phi_i  \circ   \phi_{i+1}) \circ \,d_i \\
         d_i\,\circ\phi_i&= (\phi_i  \circ   \phi_{i+1})  = \phi_{i+1}\circ \, d_i \\
         d_i\,\circ  \phi_{i+1}   & = d_i =\phi_i\circ \, d_i
     \end{align*}
     
     \item \label{item:7}  \begin{align*}
         \bar d_i\, \circ(\phi_i  \circ   \phi_{i+1}) & = \phi_i  \circ   \phi_{i+1} =(\phi_i  \circ   \phi_{i+1}) \circ\,\bar d_i \\
         \phi_i \circ \,\bar d_i &= \phi_i  \circ   \phi_{i+1}=  \bar d_i\circ \phi_{i+1}\,\\
          \phi_{i+1}\,\circ \bar d_i & = \bar d_i = \bar d_i\circ \, \phi_i  \,\, 
     \end{align*}
 
\item \begin{align*}
    ( \phi_i \circ   h_{i+1} ) \,\circ  d_i &= ( \phi_i \circ   h_{i+1} )\circ   x_i \\
  ( h_i \circ   \phi_{i+1} ) \circ  \, \bar d_i &= ( h_i \circ   \phi_{i+1} ) \circ   x_i
\end{align*}

    \item 
    \begin{align*}
       \bar d_i\circ ( \phi_i \, \circ   l_{i+1} )  &= x_i \circ ( \phi_i \, \circ   l_{i+1} )   \\
  d_i \circ  ( l_i \circ  \phi_{i+1} )  &=x_i \circ  ( l_i \circ  \phi_{i+1} )
\end{align*}
\item \label{item:14}\begin{align*}
  ( h_i \circ   \phi_{i+1} ) \circ\, d_i & = \phi_i  \circ   \phi_{i+1}=   ( \phi_i \circ h_{i+1} )\, \circ \bar d_i \\
 d_i \, \circ  ( h_i \circ   \phi_{i+1} )& =   h_i \circ   \phi_{i+1}  = \bar d_i\,\circ  ( h_i \circ   \phi_{i+1} )\\
  d_i \circ ( \phi_i \circ h_{i+1} )&= ( \phi_i \circ h_{i+1} )= \bar d_i\,\circ ( \phi_i \circ h_{i+1} )
     \end{align*}
     
\item \label{item:15}\begin{align*}
  d_i\circ(\phi_i \, \circ   l_{i+1} )  & = (\phi_i  \circ   \phi_{i+1}) = \bar d_i\circ( l_i \circ  \phi_{i+1} )\\
  (\phi_i \, \circ   l_{i+1} ) \, \circ d_i & = (\phi_i \, \circ   l_{i+1} ) = (\phi_i \, \circ   l_{i+1} ) \, \circ \bar d_i\\
  ( l_i \circ  \phi_{i+1} )\circ \bar d_i & =  l_i \circ  \phi_{i+1} =( l_i \circ  \phi_{i+1} )\circ d_i
\end{align*}

\item \label{item:16}\begin{align*}
    \bar d_i \circ x_i &= \bar d_i\circ d_i = \bar d_i \circ \bar x_i \\
     d_i \circ x_i &= d_i \circ \bar d_i = d_i \circ \bar x_i \\
    x_i \circ \bar d_i &= d_i \circ \bar d_i = \bar x_i \circ \bar d_i\\
    x_i \circ d_i &= \bar d_i \circ d_i = \bar x_i \circ d_i 
\end{align*}
\end{enumerate}

  The relations among the above generating elements are exactly those induced
by the defining relations of $\Ti$. Hence, all relations in $\mathcal{T}_n$ are understood as relations of
\emph{local} diagrams inside a fixed $n$--box, with unused positions carrying
vertical identities. See relations $(\mathcal{T}^1)-(\mathcal{T}^{18'})$ below. Moreover, a parameter
$\delta\in\mathbb C$ appears in the algebra $\delta \mathcal{T}_n$, Definition~\ref{def:reduced-tangloid-algebra}, when  a closed circle is formed,  when composing $a_iu_i$, or when  the trivial knotoid is formed by the composition $h_il_i$. See relations $(\mathcal{T}^{19})-(\mathcal{T}^{20})$ below. More precisely, the relations in $\mathcal{T}_n$ are:

\begin{align}
a_i \circ \bar x_i &= a_i = a_i\circ x_i , && 1\le i\le n-2 \tag{$\mathcal{T}^1$} \\[2mm]
 \bar x_i\circ  u_i &= u_i =  x_i\circ  u_i , && 1\le i\le n-2 \tag{$\mathcal{T}^{1'}$} \\[2mm]
a_i \circ u_{i+1}&=d_{i+1} \circ d_i , && 1\le i\le n-2 \tag{$\mathcal{T}^2$} \\[2mm]
a_{i+1}\circ u_{i}&=\bar d_{i}\circ \bar d_{i+1} , && 1\le i\le n-2\tag{$\mathcal{T}^{2'}$} \\[2mm]
\bar d_i  \circ d_i&=\phi_{i+1}, && 1\le i\le n-1\tag{$\mathcal{T}^3$} \\[2mm]
d_i \circ \bar d_i&=\phi_i,\tag{$\mathcal{T}^{3'}$} && 1\le i\le n-1 \\[2mm]
a_{i-1} &= a_i \circ (\bar d_{i+1} \circ \bar d_i \circ d_i \circ d_{i-1} \circ d_{i+1} \circ d_i ),&& 1\le i\le n-1\tag{$\mathcal{T}^4$} \\[2mm]
a_{i+1} &= a_i \circ ( d_{i-1} \circ  d_i \circ  \bar d_i \circ  \bar d_{i-1} \circ \bar d_{i+1}\circ  \bar d_i), && 1\le i\le n-1\tag{$\mathcal{T}^{4'}$} \\[2mm]
u_{i} &= (\bar d_{i+1} \circ \bar d_i \circ d_i \circ d_{i-1} \circ d_{i+1} \circ d_i ) \circ u_{i-1}&& 1\le i\le n-1\tag{$\mathcal{T}^5$} \\[2mm]
u_i &=  ( d_{i-1}\circ  d_i \circ  \bar d_i  \circ \bar d_{i-1} \circ \bar d_{i+1} \circ \bar d_i ) \circ u_{i+1}&& 1\le i\le n-1\tag{$\mathcal{T}^{5'}$} \\[2mm]
x_i\circ  \bar x_i &= \id_n = \bar x_i \circ  x_i, && 1\le i\le n-1\tag{$\mathcal{T}^6$} \\[2mm]
\bar x_i\circ  \bar x_{i+1} \circ \bar x_i &= \bar x_{i+1} \circ \bar x_i \circ \bar x_{i+1}, && 1\le i\le n-2 \tag{$\mathcal{T}^7$} \\[2mm]
\bar d_{i+1} \circ a_i \circ x_{i+1}&=d_i \circ a_{i+1}\circ \bar x_i, && 1\le i\le n-2\tag{$\mathcal{T}^8$} \\[2mm]
\bar d_{i+1} \circ a_i \circ \bar x_{i+1}&=d_i \circ a_{i+1}\circ  x_i, && 1\le i\le n-2\tag{$\mathcal{T}^{8'}$} \\[2mm]
\bar x_{i+1} \circ  u_i \circ d_{i+1} &= x_i \circ u_{i+1}\circ \bar d_i, && 1\le i\le n-2\tag{$\mathcal{T}^{9}$} \\[2mm]
\bar x_{i+1} \circ u_i \circ d_{i+1} &= x_i \circ u_{i+1} \circ \bar d_i, && 1\le i\le n-2\tag{$\mathcal{T}^{9'}$} \\[2mm]
%\end{align}
%\begin{align}
\bar d_i \circ x_{i+1}\circ d_i &= d_{i+1}\circ x_i \circ \bar d_{i+1 }, && 1\le i\le n-2\tag{$\mathcal{T}^{10}$} \\[2mm]
\bar d_i \circ \bar x_{i+1}\circ d_i &= d_{i+1}\circ \bar x_i \circ \bar d_{i+1 }, && 1\le i\le n-2\tag{$\mathcal{T}^{11}$} \\[2mm]
u_{i+1}\circ a_i &=u_i\circ a_i = u_i \circ a_{i+2}, && 1\le i\le n-1\tag{$\mathcal{T}^{12}$} \\[2mm]
h_i\circ l_{i+1} &=l_i\circ h_i = l_i \circ h_{i+1}, && 1\le i\le n-1\tag{$\mathcal{T}^{13}$} \\[2mm]
l_{i+2}\circ a_i &= l_{i+1}\circ a_i = l_i \circ a_i = l_i \circ a_{i+1}, && 1\le i\le n-1\tag{$\mathcal{T}^{14}$} \\[2mm]
u_{i+1}\circ h_i &= u_i \circ h_{i+1} = u_i\circ h_i = u_i \circ h_{i+2}, && 1\le i\le n-1\tag{$\mathcal{T}^{14'}$}\\[2mm]
a_i \circ l_i&=h_{i+1}, && 1\le i\le n-1\tag{$\mathcal{T}^{15}$}\\[2mm]
a_i \circ l_{i+1}&=h_{i}, && 1\le i\le n-1\tag{$\mathcal{T}^{15'}$} \\[2mm]
h_i\circ u_i&=l_{i+1}, && 1\le i\le n-1\tag{$\mathcal{T}^{16}$} \\[2mm]
h_{i+1} \circ u_i&=l_{i}, && 1\le i\le n-1\tag{$\mathcal{T}^{16'}$} \\[2mm]
h_{i+1}\circ d_i&=h_i, && 1\le i\le n-1\tag{$\mathcal{T}^{17}$} \\[2mm]
h_i \circ \bar d_i&=h_{i+1}, && 1\le i\le n-1\tag{$\mathcal{T}^{17'}$} \\[2mm]
\bar d_i \circ l_{i+1}&=l_i, && 1\le i\le n-1\tag{$\mathcal{T}^{18}$} \\[2mm]
d_i\circ l_i&=l_{i+1}, && 1\le i\le n-1\tag{$\mathcal{T}^{18'}$}
%a_i \circ u_i&=\delta\,id_{i-1}\otimes\emptyset_2\otimes\id_{n-i-1}, && 1\le i\le n-1\tag{$\mathcal{T}^{19}$} \\[2mm]
%h_i\,\circ  l_i &= \delta\, \id_{i-1}\otimes\emptyset\otimes\id_{n-i}\tag{$\mathcal{T}^{20}$}&& 1\le i\le n 
\end{align}
extended by the following extra relations in $\delta \mathcal{T}_n$:
\begin{align*}
    a_i \circ u_i&=\delta\,\phi_i \circ \phi_{i+1}, && 1\le i\le n-1\tag{$\mathcal{T}^{19}$} \\[2mm]
h_i\,\circ  l_i &= \delta\, \phi_i\tag{$\mathcal{T}^{20}$}
&& 1\le i\le n 
\end{align*}
\medskip
\noindent\textbf{Far commuting for disjoint support:}
Whenever two generators act on disjoint tensor blocks (e.g.\ $|i-j|>1$ for
two-strand generators), they commute; we do not list these relations separately.

\begin{definition}\label{def:regular-framed_tangloid_algebra} We call the tangloid algebra $\mathcal{T}_n$ \textit{regular tangloid algebra}, denoted $r\mathcal{T}_n$, to have the same definition as the tangloid algebra $\mathcal{T}_n$, but with the relations $(\mathcal{T}^1)$ and $(\mathcal{T}^{1'})$ omitted.
Furthermore, $\mathcal{T}_n$ gives rise to the \textit{framed tangloid category}, denoted $f\mathcal{T}_n$, having the same definition as $\mathcal{T}_n$, except for relations $(\mathcal{T}^1)$ and $(\mathcal{T}^{1'})$  which are replaced by the framing preserving relations depicted in Figure~\ref{fig:framed-ex R1}.
    
\end{definition}

% some of which may merely act as placeholders,
%%%%%%%%%%%%%%%%%%%%%%%%%

\section{The Algebra of Braidoids }\label{ss:subalgebra_of_braidoids}

We define the braidoid algebra $\mathcal{B}_n$ related to the extended braidoid category $\mathbf{Brd}_{\Ti} \subset\Ti$ of Definition~\ref{def:extended braidoid category}. The braidoid algebra $\mathcal{B}_n$ provides an algebraic structure in the set of braidoids, sought in \cite{gugumcu2017knotoids,Gugumcu2017}, filling thus a gap in the literature.

\begin{definition}[Braidoid algebra $\mathcal{B}_n$]\label{def:braidoid-subalgebra}
For $n\ge 0$ the \emph{braidoid algebra of degree $n$} is defined as the
$\mathbb{C}$–algebra
\[
  \mathcal{B}_n \, := \,\C \, \mathrm{End}_{\mathbf{Brd}_{\Ti}}(n)\,\
\]
where \[\mathrm{End}_{\mathbf{Brd}_{\Ti}}(n) := \Hom_{\mathbf{Brd}_{\Ti}}(n,n).\] 
It follows that $\mathcal{B}_n$ has as generators: $x_i$  (resp.\  $\bar x_i$) corresponding to the positive (resp.  negative) crossings    for $1\le i\le n-1$,    $h_i$ and $l_i$ the head and leg endpoints at the $i$th position, with all other strands vertical, for $1\le i\le n$, and   the diagrams obtained by inserting the empty set map $(\emptyset)$  at the $i$-th position,  $\phi_i$, with all other strands vertical, for $1\le i\le n$.  
\end{definition} 

$\mathcal{B}_n$ also includes  the 
        diagonal connectors $d_i$  joining $(i,\mathrm{bottom})$ to
        $(i{+}1,\mathrm{top})$  with all
        other strands  vertical, and $\bar d_i$  joining $(i{+}1,\mathrm{bottom})$ to
        $(i,\mathrm{top})$   with all
        other strands vertical, which are elements of special interest.  These elements are subjected to the composition rules involving the generators $\phi_i$ listed in Subsection~\ref{ss:composition rules empty}, excluding~(\ref{item:2}), (\ref{item:3}), (\ref{item:12}), (\ref{item:13}), which involve the elements $u_i$ and $a_i$, and  relations of the tangloid algebra $\mathcal{T}_n$ which do not contain the generators $u_i$ and $a_i$. For completeness we list the relations below:
\begin{align}
\bar d_i  \circ d_i&=\phi_{i+1}, && 1\le i\le n-1\tag{$\mathcal{T}^3$} \\[2mm]
d_i \circ \bar d_i&=\phi_i,\tag{$\mathcal{T}^{3'}$} && 1\le i\le n-1 \\[2mm]
x_i\circ  \bar x_i &= \id_n = \bar x_i \circ  x_i, && 1\le i\le n-1\tag{$\mathcal{T}^6$} \\[2mm]
\bar x_i\circ  \bar x_{i+1} \circ \bar x_i &= \bar x_{i+1} \circ \bar x_i \circ \bar x_{i+1}, && 1\le i\le n-2 \tag{$\mathcal{T}^7$} \\[2mm]
\bar d_i \circ x_{i+1}\circ d_i &= d_{i+1}\circ x_i \circ \bar d_{i+1 }, && 1\le i\le n-2\tag{$\mathcal{T}^{10}$} \\[2mm]
\bar d_i \circ \bar x_{i+1}\circ d_i &= d_{i+1}\circ \bar x_i \circ \bar d_{i+1 }, && 1\le i\le n-2\tag{$\mathcal{T}^{11}$} \\[2mm]
h_i\circ l_{i+1} &=l_i\circ h_i = l_i \circ h_{i+1}, && 1\le i\le n-1\tag{$\mathcal{T}^{13}$} \\[2mm]
h_{i+1}\circ d_i&=h_i, && 1\le i\le n-1\tag{$\mathcal{T}^{17}$} \\[2mm]
h_i \circ \bar d_i&=h_{i+1}, && 1\le i\le n-1\tag{$\mathcal{T}^{17'}$} \\[2mm]
\bar d_i \circ l_{i+1}&=l_i, && 1\le i\le n-1\tag{$\mathcal{T}^{18}$} \\[2mm]
d_i\circ l_i&=l_{i+1}, && 1\le i\le n-1\tag{$\mathcal{T}^{18'}$}
%h_i\,\circ  l_i &= \delta\, \id_{i-1}\otimes\emptyset\otimes\id_{n-i}\ \quad (\delta\in\mathbb C) \tag{$\mathcal{T}^{20}$}&& 1\le i\le n 
\end{align}

\begin{remark}
Elements in the braidoid algebra arising as compositions of generators are not necessarily braidoids in the geometric sense. On the other hand any braidoid can be expressed as a composition of the generators of the braidoid algebra. Our construction provides an extended algebraic framework for expressing geometric braidoids.
\end{remark}

\begin{remark}
There is a natural analogy between the braidoid framework and the inverse
braid monoid $IB_n$ of Easdown and Lavers \cite{EasdownLavers2004}.
The latter extends the braid group by allowing  strands of a braid to be
deleted, so that its elements are represented by partial braids. In the
braidoid setting, partiality arises diagrammatically
by strands with interior endpoints. Thus, at the diagrammatic level,
braidoids may be regarded as an endpoint-enhanced extension of the
partial-braid viewpoint underlying the inverse braid monoid.
 After linearisation, this analogy passes to the algebraic level: the
braidoid subalgebra of $\mathcal{T}_n$ may be compared with the monoid
algebra $\mathbb{C}[IB_n]$. A precise relationship between these two
algebras, including the existence of natural homomorphisms or embeddings,
is left for future investigation.
\end{remark}

We conclude the section with a comparison with the study in \cite{gugumcu2017knotoids,Gugumcu2017}.

\begin{remark}
In \cite{gugumcu2017knotoids,Gugumcu2017} the placeholders are called `implicit points' and generators are referred to as elementary blocks. The set of elementary blocks include the shifting blocks, corresponding to our diagonal elements, and identity elements with an empty vertical position (placeholder). But these last elements can arise  from composition of diagonal elements $d_i\bar d_i$ and $d_i\bar d_i$. 
Furthermore, in the composition rules in  \cite{gugumcu2017knotoids,Gugumcu2017},   usual ends can be concatenated with only usual ends at the same
position so that the resulting diagram contains strands stretching from top to bottom row. 
%This restriction leads to the need of including also crossing generators with an empty vertical position (placeholder). In our construction this element arises  from composition of a crossing generator with  $d_i\bar d_i$ or $d_i\bar d_i$. 
All relations listed in \cite{Gugumcu2017} are included in the relations of our braidoid algebra. For example, we can write  relation 12 in \cite{Gugumcu2017} as $\bar d_i\circ \bar d_{i+1} \circ x_{i+2} \circ d_{i+1} d_i= \bar d_i\circ d_{i+2} \circ x_{i+1} \circ d_i \circ \bar d_{i+2} $, and relation 13 as $\bar d_i\circ \bar d_{i+1} \circ x_{i+2} \circ d_{i+1} d_i=  d_{i+2}\circ d_{i+1} \circ x_{i} \circ \bar d_{i+1} \circ \bar d_{i+2}$. Our construction ensures that we have obtained a full set of relations. Finally,  we point out a misprint in relations 5 in \cite{Gugumcu2017}:  $
\sigma_i ^{\pm 1}\, l_j
=
l_j \,\bigl(\sigma_i ^ {(j)}\bigr )^{\pm 1}$. The correct one is: \[
\bigl(\sigma_i ^ {(j)}\bigr )^{\pm 1}\, l_j
=
l_j \,\sigma_i^{\pm 1},
\qquad
\text{for} \ \ i+2 \le j \le k \ \ 
\text{or} \ \ 1 \le j \le i-1.
\]
\end{remark}

%%%%%%%%%%%%%%%%%%%%%%%%%%%%%%%%%%%%%%%%%%%%%%%%%%%%%%%

\section{A diagrammatic bilinear form in the reduced tangloid algebra} \label{sec:pairing-matrix}

In this section we define the standard diagrammatic pairing  on the reduced tangloid algebra
 $\delta \mathcal{T}_n$. 
 \begin{definition}[Reduced tangloid  algebra $\mathcal{T}_n$]\label{def:reduced-tangloid-algebra}
 
 We  define the \emph{reduced tangloid algebra of degree $n$}, $\mathcal{\delta T}_n$, as the quotient of $\mathcal{T}_n$  under the extra relations where a parameter
$\delta\in\mathbb C$ appears when  a closed circle is formed,  when composing $a_iu_i$, or when  the trivial knotoid is formed by the composition $h_il_i$. 
\end{definition} 
 
 By Proposition~\ref{prop:TA-diagram-realization}, $\delta \mathcal T_n$ can be identified with the $\C$--linear span of all  equivalence classes of tangloid diagrams $n \to n$, where some points may be placeholders, representing
endomorphisms $n\to n$, modulo the extra reduced relations.

%%%%%%%%%%%%%%%%%%%%%%%%%%%%%%%%%%%%%%%%%%%%%%%%%%%%%%%%%%%%%%%%%

\subsection{The reflection anti-involution} \label{sec:star}

Let $*\colon \delta\mathcal{T}_n\to \delta\mathcal{T}_n$ denote the \emph{vertical reflection} of
diagrams in the horizontal axis of the rectangle.
Since we are in the endomorphism algebra $\End_{{\Ti}(n)}$, reflection preserves
the object $n$ and hence induces a $\C$--linear map on $\delta\mathcal{T}_n$.

\begin{lemma}[Anti-involution]\label{lem:star-antiinv}
The reflection map $*$ is an involutive anti-automorphism:
\[
(ab)^*=b^*a^*,\qquad (\id)^*=\id,\qquad (m^*)^*=m
\]
for  $a,b,m \in \delta\mathcal{T}_n$.
\end{lemma}

\begin{proof}
On the distinguished generators of Definition~\ref{def:tangloid-algebra} one has
\[
(x_i)^*=\bar x_i,\qquad (\bar x_i)^*= x_i,\qquad (u_i)^*=a_i,\qquad (a_i)^*=u_i,
\]
\[
(h_i)^*=l_i,\qquad (l_i)^*=h_i,
\qquad (\phi_i)^*=\phi_i\]
and $*$ extends multiplicatively by swapping factors and $\C$--linearly to all of $\delta\mathcal{T}_n$.
\end{proof}

\begin{definition}[Loop count]\label{def:kappa}
For a diagram $D\in\delta\mathcal{T}_n$, let $\kappa(D)$ be the number of simple closed components  
$0\to 0$ (forgetting the placeholders), including the trivial knotoid,   that appear when simplifying $D$ using the defining relations
of $\delta\mathcal{T}_n$, so that each such component contributes a scalar factor $\delta\in\mathbb C$. So, $\delta^{\,\kappa(D)}$ is the total scalar obtained by deleting all simple 
closed $0\to 0$ trivial components from $D$ including the trivial knotoid.
\end{definition}

%%%%%%%%%%%%%%%%%%%%%%%%%%%%%%%%%%%%%%%%%%%%%%%

\subsection{A diagrammatic bilinear pairing on $\delta \mathcal{T}_n$}\label{ss:bilinear pairing}

\begin{definition}[Diagrammatic pairing on $\delta \mathcal T_n$]\label{def:An-pairing} 
For any  diagrams $D_1,D_2$ in the tangloid algebra $\delta\mathcal{T}_n$ we define the pairing:
\[
\langle D_1,D_2\rangle
\;:=\;
\delta^{\,\kappa(D_2^*\,D_1)}.
\]
Here $D_2^*\, D_1\in\delta\mathcal{T}_n$ is formed by stacking $D_2^*$, the vertical reflection of $D_2$, on top of $D_1$,
 and $\kappa(\,\cdot\,)$ denotes the
number of simple closed loop components $a_iu_i$ and $h_il_i$ or their equivalents. We extend this rule $\C$--bilinearly to all of $\delta\mathcal{T}_n$. Thus, if
\[
m_1=\sum_r \alpha_r D_r,\qquad m_2=\sum_s \beta_s E_s,
\]
with $\alpha_r,\beta_s\in\C$ and $D_r,E_s$ diagrams in $\delta\mathcal{T}_n$, then
\[
\langle D_1,D_2\rangle
\;:=\;
\sum_{r,s}\alpha_r\beta_s\,\langle D_r,E_s\rangle.
\]
\end{definition}

\begin{proposition}[Basic properties]\label{prop:pairing-properties}
The pairing $\langle\cdot,\cdot\rangle$ is a symmetric $\C$--bilinear form on
$\delta\mathcal{T}_n$:
\[
\langle D_1,D_2\rangle=\langle D_2,D_1\rangle
\qquad\text{for all }D_1,D_2\in\delta\mathcal{T}_n.
\]
It also satisfies the adjointness identity
\begin{equation}\label{eq:An-adjoint}
\langle a\,D,\,b\rangle \;=\; \langle D,\,a^*b\rangle,
\qquad a,D,b\in\delta\mathcal{T}_n.
\end{equation}
\end{proposition}
\begin{proof}
    It is enough to verify the statements on basis elements, that is, diagrams,  and then extend $\C$--bilinearly.  Let $D_1,D_2\in \delta\mathcal{T}_n$ be generators.  By definition, $\langle D_1,D_2\rangle \;=\; \delta^{\,\kappa(D_2^*\,D_1)}$. Interchanging $D_1$ and $D_2$ gives the diagram $D_1^*\,D_2$ obtained by reflecting the diagram $D_2^*\,D_1$ vertically.  This reflection does not change the number of closed components. Hence
\[
\langle D_1,D_2\rangle=\langle D_2,D_1\rangle.
\]
Therefore the pairing is symmetric. Since the definition is extended linearly in each variable, it is a symmetric $\C$--bilinear form on $\delta\mathcal{T}_n$.

Now let $a,D,b\in\delta\mathcal{T}_n$ be basis diagrams. Then
\[
\langle aD,b\rangle
\]
is computed by gluing $b^*$ on top of the composite $aD$. Thus the resulting
 diagram is
\[
b^*aD.
\]
On the other hand,
\[
\langle D,a^*b\rangle
\]
is computed by gluing $(a^*b)^*$ on top of $D$. Since $*$ is an
anti-involution, we have
\[
(a^*b)^*=b^*a.
\]
Therefore, the  diagram used to compute $\langle D,a^*b\rangle$ is also
\[
b^*aD.
\]
Hence, both pairings count the same number of closed components, and so
\[
\langle aD,b\rangle=\langle D,a^*b\rangle.
\]
Extending by $\C$--linearity gives the identity for all
$a,D,b\in\delta\mathcal{T}_n$.

\end{proof}

\begin{example}
Here  examples of bilinear form in  reduced tangloid algebra show that \[
\langle D_1,D_2\rangle=\langle D_2,D_1\rangle=\delta^1.
\]
In the top example,  there is one trivial knotoid where in the bottom example there is one trivial knot. 
    \begin{figure}[H]
        \centering
\includegraphics[width=0.75\linewidth]{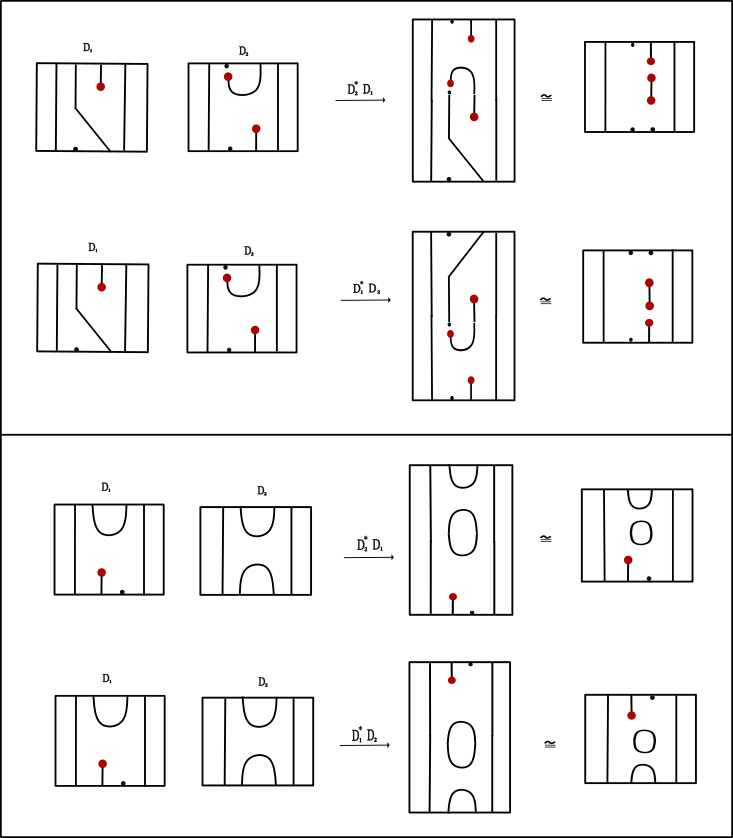}
        \caption{Example of the bilinear form in the  reduced tangloid algebra}
        \label{fig:example_bilinear_pairing}
    \end{figure}
\end{example}

%%%%%%%%%%%%%%%%%%%%%%%%%%%%%%%%%%%%%%%%%%%%%%%%%%%%%%%%%%%%%%%%%%%%%%%%

\section{Conclusion} \label{sec:conclusion}

In this paper, we have developed a diagrammatic and algebraic framework for
the study of tangloids. Tangloids include  multi-knotoids, linkoids and braidoids as special categories. We first introduced the unoriented tangloid category
$UTC$, arising from the unoriented welded tangleoid category $UWTC$ by
modifying its defining relations. This categorical construction provides a
setting in which tangloid diagrams with interior endpoints can be studied
within a strict monoidal framework.

We then introduced the extended unoriented tangloid category $\Ti$. To
construct this extension, we first adjoined the auxiliary empty morphisms
$\emptyset_t$ and $\emptyset_b$, which act as place-keepers at the top and
bottom of a diagram, respectively, and are represented diagrammatically by
points. Together with the empty generator $\emptyset$, these morphisms allow
the generators to be extended so that they can be regarded as endomorphisms
of a fixed object $n$. This extension provides the appropriate categorical
setting for defining the tangloid algebra $\mathcal{T}_n$ as the
linearisation of the endomorphism space $\operatorname{End}_{\Ti}(n)$.
 We further defined the reduced tangloid algebra
$\delta\mathcal{T}_n$ by imposing additional reduction relations under which
a trivial knot and a trivial knotoid are each replaced by the scalar
$\delta\in\mathbb{C}$.

The construction also contains natural braidoid-type structures. In
particular, the braidoid category occurs as a subcategory of $UTC$, while
the corresponding extended braidoid category is realised as a subcategory
of $\Ti$. Consequently, its linearisation gives rise to a natural braidoid
subalgebra of the tangloid algebra. This places braidoid-type structures
within the broader tangloid framework and provides an algebraic setting in
which their relationship with other diagrammatic structures can be studied.

Having established these categorical and algebraic foundations, we defined
a diagrammatic bilinear form on the reduced tangloid algebra. The pairing is
constructed using the involution induced by diagrammatic reflection together
with the closure of the resulting composed diagrams. This construction
provides a natural algebraic structure associated with the diagrammatic
calculus and opens the way to the study of degeneracy, radicals, and
representations of tangloid algebras.

Several directions remain for future investigation. One natural direction
is the introduction of a Temperley--Lieb tangloid algebra
$\TL\mathcal{T}_n$ and the study of its relationship with
$\mathcal{T}_n$ and $\delta\mathcal{T}_n$. More generally, it would be
interesting to formulate analogues of classical diagrammatic algebras within
the tangloid framework and to determine how their defining generators and
relations arise from tangloid diagrams. This is subject of sequel work.

We shall also investigate further the extended braidoid category framework in connection with the  inverse braid monoid. At the algebraic level, this suggests a natural extension of the monoid algebra $\mathbb{C}[IB_n]$ by the braidoid algebra $\mathcal{B}_n$. 

Another direction is the development
of a representation theory for tangloid algebras, including the construction
and study of potentially non-unital representations. These developments may
clarify the connections between tangloid algebras, classical diagrammatic
algebras, and their associated representation-theoretic structures.

Finally, an oriented version of the construction may be obtained by introducing
oriented counterparts of the generators and adapting the defining relations
to respect the orientations. The corresponding oriented tangloid algebras
and their representation theory provide a natural direction for future work.

\section*{Acknowledgements}
This project was funded by the Deanship of Scientific Research (DSR) at King Abdulaziz University, Jeddah, Saudi Arabia, under grant no. IPP-156-665-2025. The first author gratefully acknowledges the DSR for technical and financial support, and  thanks Prof. Ahmad Alghamdi for valuable support and insightful discussions. The second author gratefully acknowledges hospitality and support of the Institut Henri Poincaré (UAR 839 CNRS-Sorbonne Université), and LabEx CARMIN (ANR-10-LABX-59-01), and also  discussions with Louis H. Kauffman and with Ganna Kudryavtseva. 

\section*{Conflict of interest}
The authors declare no potential conflicts of interest.

% -------------------- BIBLIOGRAPHY --------------------

\end{document}